\documentclass[a4paper,11pt,leqno]{article}%
\usepackage{graphicx}
\usepackage{amsmath}
\usepackage{amsfonts}
\usepackage{amssymb}%
\newtheorem{theorem}{Theorem}

\newtheorem{corollary}[theorem]{Corollary}

\newtheorem{definition}{Definition}
\newtheorem{example}{Example}

\newtheorem{lemma}[theorem]{Lemma}
\newtheorem{notation}{Notation}

\newtheorem{proposition}[theorem]{Proposition}
\newtheorem{remark}{Remark}

\newenvironment{proof}[1][Proof]{\textbf{#1.} }{\ \rule{0.5em}{0.5em}}
\begin{document}

\title{Microlocal Asymptotic Analysis in Algebras of Generalized Functions}
\author{Antoine Delcroix\\Equipe Analyse Alg\'{e}brique Non Lin\'{e}aire -- \textit{Laboratoire A O C}\\Facult\'{e} des sciences - Universit\'{e} des Antilles et de la Guyane,\\BP 250, 97157 Pointe \`{a} Pitre Cedex, Guadeloupe (France) \\Antoine.Delcroix@univ-ag.fr
\and Jean-Andr\'{e} Marti\\Equipe Analyse Alg\'{e}brique Non Lin\'{e}aire -- \textit{Laboratoire G T S I}\\Facult\'{e} des sciences -- Universit\'{e} des Antilles et de la Guyane,\\BP 250, 97157 Pointe \`{a} Pitre Cedex, Guadeloupe (France) \\Jean-Andre.Marti@univ-ag.fr
\and Michael Oberguggenberger\thanks{Supported by FWF (Austria), grant Y237.}\\Institute of Basic Sciences in Engineering \ -- Unit of Engineering Mathematics\\Faculty of Civil Engineering -- University of Innsbruck, \\Technikerstra\ss e 13 -- 6020 Innsbruck, Austria \\Michael.Oberguggenberger@uibk.ac.at}
\maketitle

\begin{abstract}
We introduce a new type of local and microlocal asymptotic analysis in
algebras of generalized functions, based on the presheaf properties of those
algebras and on the properties of their elements with respect to a
regularizing parameter. Contrary to the more classical frequential analysis
based on the Fourier transform, we can describe a singular asymptotic spectrum
which has good properties with respect to nonlinear operations. In this spirit
we give several examples of propagation of singularities through nonlinear operators.

\end{abstract}

\medskip

\medskip

%\noindent\textbf{Running headline:} Microlocal Asymptotic Analysis

\section{Introduction}

Various nonlinear theories of generalized functions have been developed over
the past twenty years, with contributions by many authors. These theories have
in common that the space of distributions is enlarged or embedded into
algebras so that nonlinear operations on distributions become possible. These
methods have been especially efficient in formulating and solving nonlinear
differential problems with irregular data.\smallskip

Most of the algebras of generalized functions possess the structure of sheaves
or presheaves, which may contain some sub(pre)sheaves with particular
properties. For example, the sheaf $\mathcal{G}$ of the special Colombeau
algebras \cite{JFC3,GKOS,NePiSc} contains the subsheaf $\mathcal{G}^{\infty}$
of so-called regular sections of $\mathcal{G}$ such that the embedding:
$\mathcal{G}^{\infty}\rightarrow\mathcal{G}$ is the natural extension of the
classical one: $\mathrm{C}^{\infty}\rightarrow\mathcal{D}^{\prime}$. This
notion of regularity leads to $\mathcal{G}^{\infty}$-local or microlocal
analysis of generalized functions, extending the classical results on the
$\mathrm{C}^{\infty}$-microlocal analysis of distributions due to
H\"{o}rmander \cite{HorPDOT1}. This concept has been slightly extended in
\cite{ADRReg} to less restrictive kinds of measuring regularity. In
\cite{JAM3}, microlocal regularity theory in analytic and Gevrey classes has
been generalized to algebras of generalized functions. Many results on
propagation of singularities and pseudodifferential techniques have been
obtained during the last years (see
\cite{GaGrOb,GarHorm,HorDeH,HorKun,HorObPil}). Nevertheless, these results are
still mainly limited to linear cases, since they use frequential methods based
on the Fourier transform.\smallskip

In this paper, we develop a new type of asymptotic local and microlocal
analysis of generalized functions in the framework of $(\mathcal{C}%
,\mathcal{E},\mathcal{P})$-algebras \cite{JAM0,JAM1}, following first steps
undertaken in \cite{JAM0}. An example of the construction is given by taking
$\mathcal{G}$ as a special case of a $(\mathcal{C},\mathcal{E},\mathcal{P}%
)$-structure (see Subsection \ref{SPSSubSecDC} for details). Let $\mathcal{F}$
be a subsheaf of vector spaces (or algebras) of $\mathcal{G}$ and
$(u_{\varepsilon})_{\varepsilon}$ a representative of $u\in\mathcal{G}\left(
\Omega\right)  $ for some open set $\Omega\subset\mathbb{R}^{n}$. We first
define $\mathcal{O}_{\mathcal{G}}^{\mathcal{F}}\left(  u\right)  $ as the set
of all $x\in\Omega$ such that $u_{\varepsilon}$ tends to a section of
$\mathcal{F}$ above some neighborhood of $x$. The $\mathcal{F}$-singular
support of $u$ is $\Omega\backslash\mathcal{O}_{\mathcal{G}}^{\mathcal{F}%
}\left(  u\right)  $. For fixed $x$ and $u$, $N_{x}(u)$ is the set of all
$r\in\mathbb{R}_{+}$ such that $\varepsilon^{r}u_{\varepsilon}$ tends to a
section of $\mathcal{F}$ above some neighborhood of $x$. The $\mathcal{F}%
$-singular spectrum of $u$ is the set of all $\left(  x,r\right)  \in
\Omega\times\mathbb{R}_{+}$ such that $r\in\mathbb{R}_{+}\backslash N_{x}(u)$.
It gives a spectral decomposition of the $\mathcal{F}$-singular support of
$u$.\smallskip

This asymptotic analysis is extended to $(\mathcal{C},\mathcal{E}%
,\mathcal{P})$-algebras. This gives the general asymptotic framework, in which
the net $\left(  \varepsilon^{r}\right)  _{\varepsilon}$ is replaced by any
net $a$ satisfying some technical conditions, leading to the concept of the
$(a,\mathcal{F})$-singular asymptotic spectrum.\ The main advantage is that
this asymptotic analysis is compatible with the algebraic structure of the
$(\mathcal{C},\mathcal{E},\mathcal{P})$-algebras.\ Thus, the $(a,\mathcal{F}%
)$-singular asymptotic spectrum inherits good properties with respect to
nonlinear operations (Theorem \ref{SPSAFSASProd} and Corollary
\ref{SPSAFSASPow}).\smallskip

The paper is organized as follows. In Section \ref{SecPDLPA}, we introduce the
sheaves of $(\mathcal{C},\mathcal{E},\mathcal{P})$-algebras and develop the
local asymptotic analysis. Section \ref{SPSMicAn} is devoted to the
$(a,\mathcal{F})$\textit{-}microlocal analysis and specially to the nonlinear
properties of the $(a,\mathcal{F})$-singular asymptotic spectrum.\ In Section
\ref{SPSPDE} various examples of the propagation of singularities through non
linear differential operators are given.

\section{Preliminary definitions and local parametric analysis\label{SecPDLPA}%
}

\subsection{The presheaves of $(\mathcal{C},\mathcal{E},\mathcal{P}%
)$-algebras: the algebraic structure \label{SPSSubSecCEP}}

%

%TCIMACRO{\TeXButton{Compteurequation}{\setcounter{equation}{6}}}%
%BeginExpansion
\setcounter{equation}{6}%
%EndExpansion

We begin by recalling the notions from \cite{JAM0, JAM1} that form the basis
for our study.\smallskip

\noindent(a) Let:

\noindent$\left(  1\right)  $~$\Lambda$ be a set of indices;

\noindent$\left(  2\right)  ~A$ be a solid subring of the ring $\mathbb{K}%
^{\Lambda}$\ ($\mathbb{K=R}$ or $\mathbb{C}$); this means that whenever
$(\left\vert s_{\lambda}\right\vert )_{\lambda}\leq(\left\vert r_{\lambda
}\right\vert )_{\lambda}$\ for some $\left(  (s_{\lambda})_{\lambda
},(r_{\lambda})_{\lambda}\right)  \in\mathbb{K}^{\Lambda}\times A$, that is,
$\left\vert s_{\lambda}\right\vert \leq\left\vert r_{\lambda}\right\vert $ for
all $\lambda$, it follows that $(s_{\lambda})_{\lambda}\in A\,$;

\noindent$\left(  3\right)  $~$I_{A}$ be a solid ideal of $A\,$;

\noindent$\left(  4\right)  $~$\mathcal{E}$ be a sheaf of $\mathbb{K}%
$-topological algebras over a topological space $X\,$.

Moreover, suppose that

\noindent$\left(  5\right)  $~for any open set $\Omega$ in $X$, the algebra
$\mathcal{E}(\Omega)$ is endowed with a family $\mathcal{P}(\Omega
)=(p_{i})_{i\in I(\Omega)}$ of semi-norms such that if $\Omega_{1}$,
$\Omega_{2}$ are two open subsets of $X$ with $\Omega_{1}\subset$ $\Omega_{2}%
$, it follows that $I(\Omega_{1})\subset I(\Omega_{2})$ and if $\rho_{1}^{2}$
is the restriction operator $\mathcal{E}(\Omega_{2})\rightarrow\mathcal{E}%
(\Omega_{1})$, then, for each $p_{i}\in\mathcal{P}(\Omega_{1})$ the semi-norm
$\widetilde{p}_{i}=p_{i}\circ\rho_{1}^{2}$ extends $p_{i}$ to $\mathcal{P}%
(\Omega_{2})\,$.

\noindent$\left(  6\right)  $~Let $\Theta=(\Omega_{h})_{h\in H}$ be any family
of open sets in $X$ with $\Omega=\cup_{h\in H}\Omega_{h}.$ Then, for each
$p_{i}\in\mathcal{P}(\Omega)$, $i\in I(\Omega)$, there exist a finite
subfamily of $\Theta$: $\Omega_{1},\ldots,\ \Omega_{n\left(  i\right)  }$ and
corresponding semi-norms $p_{1}\in\mathcal{P}(\Omega_{1}),\ldots,\ p_{n\left(
i\right)  }\in\mathcal{P}(\Omega_{n\left(  i\right)  })$, such that, for any
$u\in\mathcal{E}(\Omega)$%
\[
p_{i}\left(  u\right)  \leq p_{1}\left(  u\left\vert _{\Omega_{1}}\right.
\right)  +\ldots+p_{n\left(  i\right)  } ( u\, \vert_{\Omega_{n\left(
i\right)  }} ) .
\]

\noindent(b) Define $\left\vert B\right\vert =\left\{  \left(  \left\vert
r_{\lambda}\right\vert \right)  _{\lambda},\ \left(  r_{\lambda}\right)
_{\lambda}\in B\right\}  $, $B=A$ or $I_{A}$, and set
\begin{gather*}
\mathcal{H}_{(A,\mathcal{E},\mathcal{P})}(\Omega)=\left\{  (u_{\lambda
})_{\lambda}\in\left[  \mathcal{E}(\Omega)\right]  ^{\Lambda}\mid\forall i\in
I(\Omega),\left(  (p_{i}(u_{\lambda})\right)  _{\lambda}\in\left\vert
A\right\vert \right\} \\
\mathcal{J}_{(I_{A},\mathcal{E},\mathcal{P})}(\Omega)=\left\{  (u_{\lambda
})_{\lambda}\in\left[  \mathcal{E}(\Omega)\right]  ^{\Lambda}\mid\forall i\in
I(\Omega),\left(  p_{i}(u_{\lambda})\right)  _{\lambda}\in\left\vert
I_{A}\right\vert \right\} \\
\mathcal{C}=A/I_{A},
\end{gather*}

Note that, from (2), $\left\vert A\right\vert $ is a subset of $A$ and that
$A_{+}=\{(b_{\lambda})_{\lambda}\in A,\,\forall\lambda\in\Lambda,\ b_{\lambda
}\geq0\} =\left\vert A\right\vert $. The same holds for $I_{A}$. Furthermore,
(2) implies also that $A$ is a $\mathbb{K}$-algebra. Indeed, it suffices to
show that $A$ is stable under multiplication by elements of $\mathbb{K}$. Let
$c$ be in $\mathbb{K}$ and $(a_{\lambda})_{\lambda}\in A$. Then $(ca_{\lambda
})_{\lambda}$ satisfies $(\vert ca_{\lambda}\vert)_{\lambda}\leq(\vert
na_{\lambda}\vert)_{\lambda}$ for some $n\in\mathbb{N}$.\ We have
$(na_{\lambda})_{\lambda}\in A$ since $A$ is stable under addition. Thus,
using (2), we get that $(ca_{\lambda})_{\lambda}\in A$.

For later reference, we recall the following notions entering in the
definition of a sheaf ${\mathcal{A}}$ on $X$. Let $(\Omega_{h})_{h\in H}$ be a
family of open sets in $X$ with $\Omega= \cup_{h\in H}\Omega_{h}.$

\begin{itemize}
\item[$(F_{1})$] (Localization principle) Let $u, v \in{\mathcal{A}}(\Omega)$.
If all restrictions $u\vert_{\Omega_{h}}$ and $u\vert_{\Omega_{h}}$, $h\in H$,
coincide, then $u = v$ in ${\mathcal{A}}(\Omega)$.

\item[$(F_{2})$] (Gluing principle) Let $(u_{h})_{h\in H}$ be a coherent
family of elements of ${\mathcal{A}}(\Omega_{h})$, that is, the restrictions
to the non-void intersections of the $\Omega_{h}$ coincide. Then there is an
element $u\in{\mathcal{A}}(\Omega)$ such that $u|_{\Omega_{h}}=u_{h}$ for all
$h\in H$.
\end{itemize}

\begin{proposition}
\label{SPSPropSubAlg} \textit{(i)~}$\mathcal{H}_{(A,\mathcal{E},\mathcal{P})}%
$\textit{\ is a sheaf of }$\mathbb{K}$-\textit{subalgebras of the sheaf
}$\mathcal{E}^{\Lambda}$\textit{;\newline(ii)~}$\mathcal{J}_{(I_{A}%
,\mathcal{E},\mathcal{P})}$\textit{\ is a sheaf of ideals of }$\mathcal{H}%
_{(A,\mathcal{E},\mathcal{P})}$\textit{.}
\end{proposition}

\begin{proof}
The proof can be found in~\cite{JAM0, JAM1}, so we just recall the main steps.
We start from the statement that $\mathcal{E}$ and $\mathcal{E}^{\Lambda}$ are
already sheaves of algebras. From (5), we infer that $\mathcal{H}%
_{(A,\mathcal{E},\mathcal{P})}$ and $\mathcal{J}_{(I_{A},\mathcal{E}%
,\mathcal{P})}$ are a presheaves (the restriction property holds) and that the
localization property $(F_{1})$ is valid. To obtain the gluing property
$(F_{2})$ we need property (6), which generalizes the situation from
$C^{\infty}$ to $\mathcal{E}$.
\end{proof}

\begin{theorem}
The factor $\mathcal{H}_{(A,\mathcal{E},\mathcal{P})}/\mathcal{J}%
_{(I_{A},\mathcal{E},\mathcal{P})}$ is a presheaf satisfying the localization
principle $\left(  F_{1}\right)  $.
\end{theorem}

\begin{proof}
From the previous proposition, we know that $\mathcal{A}=\mathcal{H}%
_{(A,\mathcal{E},\mathcal{P})}/\mathcal{J}_{(I_{A},\mathcal{E},\mathcal{P})}$
is a presheaf. For $\Omega_{1}\subset$\ $\Omega_{2}$, the restriction is
defined by%
\[%
\begin{array}
[c]{ccc}%
\mathcal{A}\left(  \Omega_{2}\right)  & \overset{\mathcal{R}_{1}^{2}%
}{\longrightarrow} & \mathcal{A}\left(  \Omega_{1}\right) \\
u & \longmapsto & u\left\vert _{\Omega_{1}}\right.  :=\left[  u_{\lambda
}\left\vert _{\Omega_{1}}\right.  \right]
\end{array}
\]
where $\left(  u_{\lambda}\right)  _{\lambda}$ is any representative of
$u\in\mathcal{A}(\Omega_{2})$ and $\left[  u_{\lambda}\mid_{\Omega_{1}%
}\right]  $ denotes the class of $\left(  u_{\lambda}\mid_{\Omega_{1}}\right)
_{\lambda}$. The definition is consistent and independent of the
representative because for each $\left(  u_{\lambda}\right)  _{\lambda
\in\Lambda}\in\mathcal{H}_{(A,\mathcal{E},\mathcal{P})}(\Omega_{2})$ and
$\left(  \eta_{\lambda}\right)  _{\lambda\in\Lambda}\in\mathcal{J}%
_{(I_{A},\mathcal{E},\mathcal{P})}(\Omega_{2})$, we have
\[
\left(  u_{\lambda}\right)  _{\lambda}\left\vert _{\Omega_{1}}\right.
:=\left(  u_{\lambda}\left\vert _{\Omega_{1}}\right.  \right)  _{\lambda}%
\in\mathcal{H}_{(A,\mathcal{E},\mathcal{P})}(\Omega_{1})\ ,\ \ \ \left(
\eta_{\lambda}\right)  _{\lambda}\left\vert _{\Omega_{1}}\right.  :=\left(
\eta_{\lambda}\left\vert _{\Omega_{1}}\right.  \right)  _{\lambda}%
\in\mathcal{J}_{(I_{A},\mathcal{E},\mathcal{P})}(\Omega_{1})
\]
The localization principle is also obviously fulfilled because $\mathcal{J}%
_{(I_{A},\mathcal{E},\mathcal{P})}$ is itself a sheaf.
\end{proof}

\begin{proposition}
\label{SPSPropConst}\textit{Under the hypothesis (2), the constant sheaf
}$\mathcal{H}_{(A,\mathbb{K},\left\vert .\right\vert )}/\mathcal{J}%
_{(I_{A},\mathbb{K},\left\vert .\right\vert )}$\textit{\ is exactly the ring
}$\mathcal{C}=A/I_{A}.$
\end{proposition}

\begin{proof}
We clearly have $\mathcal{H}_{(A,\mathbb{K},\left\vert .\right\vert )}=A$ and
$\mathcal{J}_{(I_{A},\mathbb{K},\left\vert .\right\vert )}=I_{A}.$
\end{proof}

\begin{definition}
\label{SPSdefCEP} \textit{\textit{The factor presheaf of algebras }}over the
ring $\mathcal{C}=A/I_{A}$:%
\[
\mathcal{A}=\mathcal{H}_{(A,\mathcal{E},\mathcal{P})}/\mathcal{J}%
_{(I_{A},\mathcal{E},\mathcal{P})}%
\]
is called a presheaf of $(\mathcal{C},\mathcal{E},\mathcal{P})$%
\textit{\textit{-algebras.}}
\end{definition}

\begin{notation}
\textit{We denote by }$\left[  u_{\lambda}\right]  $\textit{\ the class in
}$\mathcal{A}(\Omega)$ \textit{defined by }$\left(  u_{\lambda}\right)
_{\lambda\in\Lambda}\in\mathcal{H}_{(A,\mathcal{E},\mathcal{P})}(\Omega)$. For
$u\in\mathcal{A}$, the notation $\left(  u_{\lambda}\right)  _{\lambda
\in\Lambda}\in u$ means that $\left(  u_{\lambda}\right)  _{\lambda\in\Lambda
}$ is a representative of $u$.
\end{notation}

\begin{remark}
The problem of rendering $\mathcal{A}$ a sheaf (and even a fine sheaf) is not
studied here. It is well known that the Colombeau algebra $\mathcal{G}$, which
is a special case of a $(\mathcal{C},\mathcal{E},\mathcal{P})$%
\textit{\textit{-algebra}} (see Subsection \ref{SPSSubSecDC}), forms a fine
sheaf \cite{AraBia,GKOS}. The sheaf property can be inferred from the
existence of a \textrm{C}$^{\infty}$-partition of unity associated to any open
covering of an open set $\Omega$ of $\mathbb{R}^{d}$. This existence is
fulfilled because $X=\mathbb{R}^{d}$ is a locally compact Hausdorff space. On
the other hand, \textrm{C}$^{\infty}$ is a fine sheaf because multiplication
by a smooth function defines a sheaf homomorphism in a natural way. Hence the
usual topology and \textrm{C}$^{\infty}$-partition of unity defines the
required sheaf partition of unity. Observing that $\mathcal{G}$ is a sheaf of
\textrm{C}$^{\infty}$-modules and using the well known result that a sheaf of
modules on a fine sheaf is itself a fine sheaf, we obtain the corresponding
assertion about $\mathcal{G}$. In the general case, turning $\mathcal{A}$ into
a sheaf requires additional hypotheses, which are not necessary for the
results in this paper. Indeed, the presheaf structure of $\mathcal{A}$ and the
$\left(  F_{1}\right)  $-principle are sufficient to develop our local and
microlocal asymptotic analysis.
\end{remark}

\begin{remark}
\label{SPSRemMorphism}The map $\iota:\mathbb{K}\rightarrow A$ defined by
$\iota\left(  r\right)  =\left(  r\right)  _{\lambda}$ is an embedding of
algebras and induces a ring morphism from $\mathbb{K}\rightarrow\mathcal{C}$
if, and only if, $A$ is unitary (Lemma 14, \cite{JAM1}). Indeed, if $A$ is
unitary, $\left(  r\right)  _{\lambda}=r\left(  1_{\lambda}\right)  _{\lambda
}$ is an element of $A$ since $A$ is a $\mathbb{K}$-algebra, and $\iota$ is
clearly an injective ring morphism. The converse is obvious. Moreover, if
$\Lambda$ is a directed set with partial order relation $\prec$ and if
\begin{equation}
I_{A}\subset\left\{  \left(  a_{\lambda}\right)  _{\lambda}\in A\,\vert
\,\lim_{\Lambda}a_{\lambda}=0\right\}  , \label{SPSCondAnneau}%
\end{equation}
then the morphism $\iota$ is injective. Indeed, if $\left[  \iota\left(
r\right)  \right]  =0$, relation (\ref{SPSCondAnneau}) implies that the limit
of the constant sequence $\left(  r\right)  _{\lambda}$ is null, thus $r=0$.
\end{remark}

\subsection{Relationship with distribution theory and Colombeau
algebras\label{SPSSubSecDC}}

One main feature of this construction is that we can choose the triple
$(\mathcal{C},\mathcal{E},\mathcal{P})$ such that the sheaves $\mathrm{C}%
^{\infty}$ and $\mathcal{D}^{\prime}$ are embedded in the corresponding sheaf
$\mathcal{A}$. In particular, we can multiply (the images of) distributions in
$\mathcal{A}$.

We consider the sheaf $\mathcal{E}=\mathrm{C}^{\infty}$ over $\mathbb{R}^{d}$,
where $\mathcal{P}$ is the usual family of topologies $\left(  \mathcal{P}%
_{\Omega}\right)  _{\Omega\in\mathcal{O}\left(  \mathbb{R}^{d}\right)  }$.
Here $\mathcal{O}\left(  \mathbb{R}^{d}\right)  $ denotes the set of all open
sets of $\mathbb{R}^{d}$; this notation will be used in the sequel. Let us
recall that $\mathcal{P}_{\Omega}$ is defined by the family of semi-norms
$\left(  p_{K,l}\right)  _{K\Subset\Omega,l\in\mathbb{N}}$ with
\[
\forall f\in\mathrm{C}^{\infty}\left(  \Omega\right)  ,\ \ \ p_{K,l}\left(
f\right)  =\sup_{x\in K,\left\vert \alpha\right\vert \leq l}\left\vert
\partial^{\alpha}f\left(  x\right)  \right\vert .
\]
From Lemma 14 in \cite{JAM1}, it follows that the canonical maps, defined for
any $\Omega\in\mathcal{O}\left(  \mathbb{R}^{d}\right)  $ by
\[
\sigma_{\Omega}:\mathrm{C}^{\infty}\left(  \Omega\right)  \rightarrow
\mathcal{H}_{(A,\mathcal{E},\mathcal{P})}(\Omega)\ \ \ \ \ f\mapsto\left(
f\right)  _{\lambda}\text{,}%
\]
are injective morphism\ of algebras if, and only if, $A$ is unitary. Under
this assumption, these maps give rise to a canonical sheaf embedding of
$\mathrm{C}^{\infty}$ into $\mathcal{H}_{(A,\mathcal{E},\mathcal{P})}$ and
(using a partition of unity in $\mathrm{C}^{\infty}$ inducing a sheaf
structure on $\mathcal{A}$) to a canonical sheaf morphism of algebras from
$\mathrm{C}^{\infty}$ into $\mathcal{A}$. This sheaf morphism turns out to be
a sheaf morphism of embeddings if $\Lambda$ is a directed set with respect to
a partial order $\prec$ and if relation (\ref{SPSCondAnneau}) holds.\smallskip

We shall address the question of the embedding of $\mathcal{D}^{\prime}$ for
the simple case of $\Lambda=\left(  0,1\right]  $. For a net $\left(
\varphi_{\varepsilon}\right)  _{\varepsilon}$ of mollifiers given by
\[
\varphi_{\varepsilon}\left(  x\right)  =\dfrac{1}{\varepsilon^{d}}%
\varphi\left(  \dfrac{x}{\varepsilon}\right)  ,\ x\in\mathbb{R}^{d}\text{
\ where }\varphi\in\mathcal{D}( \mathbb{R}^{d}) \text{ and\ }%
%TCIMACRO{\tint }%
%BeginExpansion
{\textstyle\int}
%EndExpansion
\varphi\left(  x\right)  dx=1,
\]
and $T\in\mathcal{D}^{\prime}\left(  \mathbb{R}^{d}\right)  $, the net
$\left(  T\ast\varphi_{\varepsilon}\right)  _{\varepsilon}$ is a net of smooth
functions in \textrm{C}$^{\infty}\left(  \mathbb{R}^{d}\right)  $, moderately
increasing in $\dfrac{1}{\varepsilon}$.\ This means that
\begin{equation}
\forall K\Subset\mathbb{R}^{d},\forall l\in N,\ \ \exists m\in\mathbb{N}%
\text{\ :\ \ }p_{K,l}\mathbb{\ }\left(  T\ast\varphi_{\varepsilon}\right)
=\mathrm{o}(\varepsilon^{-m}),\ \text{as }\varepsilon\rightarrow0.
\label{SPSPolGro1}%
\end{equation}
This justifies to choose
\begin{align*}
A  &  =\left\{  \left(  r_{\varepsilon}\right)  _{\varepsilon}\in
\mathbb{R}^{\left(  0,1\right]  }\;\mid\exists m\in\mathbb{N\ }:\mathbb{\ }%
\left\vert u_{\varepsilon}\right\vert =\mathrm{o}(\varepsilon^{-m}),\ \text{as
}\varepsilon\rightarrow0\right\} \\
I  &  =\left\{  \left(  r_{\varepsilon}\right)  _{\varepsilon}\in
\mathbb{R}^{\left(  0,1\right]  }\mid\forall q\in\mathbb{N\ }:\mathbb{\ }%
\left\vert u_{\varepsilon}\right\vert =\mathrm{o}(\varepsilon^{q}),\ \text{as
}\varepsilon\rightarrow0\right\}  .
\end{align*}
In this case (with $\mathcal{E}=\mathrm{C}^{\infty}$), the sheaf of algebras
$\mathcal{A}=\mathcal{H}_{(A,\mathcal{E},\mathcal{P})}/\mathcal{J}%
_{(I_{A},\mathcal{E},\mathcal{P})}$ is exactly the so-called special Colombeau
algebra $\mathcal{G}$ \cite{JFC3,GKOS,Ober1}. Then, for all $\Omega
\in\mathcal{O}\left(  \mathbb{R}^{d}\right)  $, $\mathrm{C}^{\infty}\left(
\Omega\right)  $ is embedded in $\mathcal{A}\left(  \Omega\right)  $ by
\[
\sigma_{\Omega}:\mathrm{C}^{\infty}(\Omega)\rightarrow\mathcal{A}%
(\Omega)\ \ \ \ \ f\mapsto\left[  f_{\varepsilon}\right]  \text{ with
}f_{\varepsilon}=f\text{ for all }\varepsilon\text{ in }\left(  0,1\right]  ,
\]
because the constant net $\left(  f\right)  _{\varepsilon}$ belongs to
$\mathcal{H}_{(A,\mathcal{E},\mathcal{P})}\left(  \mathbb{R}^{d}\right)  $ and
$\left(  f\right)  _{\varepsilon}\in\mathcal{J}_{(I_{A},\mathcal{E}%
,\mathcal{P})}$ implies $f=0$ in $\mathrm{C}^{\infty}(\Omega)$. Furthermore,
$\mathcal{D}^{\prime}\left(  \mathbb{R}^{d}\right)  $ is embedded in
$\mathcal{A}\left(  \mathbb{R}^{d}\right)  $ by the mapping
\[
\iota:T\mapsto\left(  T\ast\varphi_{\varepsilon}\right)  _{\varepsilon}%
\]
Indeed, relation (\ref{SPSPolGro1}) implies that $\left(  T\ast\varphi
_{\varepsilon}\right)  _{\varepsilon}$ belongs to $\mathcal{H}_{(A,\mathcal{E}%
,\mathcal{P})}\left(  \mathbb{R}^{d}\right)  $ and $\left(  T\ast
\varphi_{\varepsilon}\right)  _{\varepsilon}\in\mathcal{J}_{(I_{A}%
,\mathcal{E},\mathcal{P})}$ implies that $T\ast\varphi_{\varepsilon
}\rightarrow0$ in $\mathcal{D}^{\prime}\left(  \mathbb{R}^{d}\right)  $, as
$\varepsilon\rightarrow0$ and $T=0$. Thus,\ $\iota$ is a well defined
injective map.\smallskip

With the help of cutoff functions, we can define analogously, for each open
set $\Omega$ in $\mathbb{R}^{d}$, an embedding $\iota_{\Omega}$ of
$\mathcal{D}^{\prime}\left(  \Omega\right)  $ into $\mathcal{A}\left(
\Omega\right)  $, and finally a sheaf embedding $\mathcal{D}^{\prime
}\rightarrow\mathcal{A}$. This embedding depends on the choice of the net of
mollifiers $\left(  \varphi_{\varepsilon}\right)  _{\varepsilon}$. We refer
the reader to \cite{ADEmbed,NePiSc} for more complete discussions about
embeddings in Colombeau's case and to \cite{JAM1} for the case of
$(\mathcal{C},\mathcal{E},\mathcal{P})$-algebras.

\subsection{An association process\label{SPSSubsecAssoc}}

We return to the general case with the assumption that $A$ is unitary and
$\Lambda$ is a directed set with partial order relation $\prec.\smallskip$

Let us denote by:\vspace{-0.04in}

\begin{itemize}
\item $\Omega$ an open subset of $X$,\vspace{-0.04in}

\item $\mathcal{F}$ a given sheaf (or presheaf) of topological $\mathbb{K}%
$-vector spaces (resp. $\mathbb{K}$-algebras) over $X$ containing
$\mathcal{E}$ as a subsheaf of topological algebras,\vspace{-0.04in}

\item $a$ a map from $\mathbb{R}_{+}$ to $A_{+}$ such that $a(0)=1$ (for
$r\in\mathbb{R}_{+}$, we denote $a\left(  r\right)  $ by $\left(  a_{\lambda
}\left(  r\right)  \right)  _{\lambda}$).\vspace{-0.04in}
\end{itemize}

In the Colombeau case, a typical example would be $a_{\varepsilon}(r) =
\varepsilon^{r}$, $\varepsilon\in(0,1]$.

For $\left(  v_{\lambda}\right)  _{\lambda}\in\mathcal{H}_{(A,\mathcal{E}%
,\mathcal{P})}\left(  \Omega\right)  $, we shall denote the limit of $\left(
v_{\lambda}\right)  _{\lambda}$ for the $\mathcal{F}$\textit{-}topology by
$\lim\limits_{\Lambda}\left.  _{\mathcal{F}(\Omega)}\right.  v_{\lambda}$ when
it exists. We recall that $\lim\limits_{\Lambda}\left.  _{\mathcal{F}%
(V)}\right.  u_{\lambda}\left\vert _{V}\right.  =f\in\mathcal{F}(V)$ iff, for
each $\mathcal{F}$\textit{-}neighborhood $W$\ of $f$, there exists
$\lambda_{0}\in\Lambda$\ such that%
\[
\lambda\prec\lambda_{0}\ \Longrightarrow\ u_{\lambda}\vert_{V}\in W.
\]
We suppose also that we have, for each open subset $V\subset\Omega$,%
\begin{equation}
\mathcal{J}_{(I_{A},\mathcal{E},\mathcal{P})}(V)\subset\left\{  \left(
v_{\lambda}\right)  _{\lambda}\in\mathcal{H}_{(A,\mathcal{E},\mathcal{P}%
)}(V):\lim\limits_{\Lambda}\left.  _{\mathcal{F}(V)}\right.  v_{\lambda
}=0\right\}  . \label{SPSHypoth1}%
\end{equation}

\begin{definition}
\label{SPSdefAssoc}Consider $u=\left[  u_{\lambda}\right]  \in\mathcal{A}%
(\Omega)$, $r\in\mathbb{R}_{+}$, $V$ an open subset of $\Omega$ and
$f\in\mathcal{F}(V)$.\ We say that $u$ is $a\left(  r\right)  $%
\emph{-associated with }$f$\emph{ in }$V$:%
\[
u\overset{a(r)}{\underset{\mathcal{F}\left(  V\right)  }{\sim}}f
\]
if $\lim\limits_{\Lambda}\left.  _{\mathcal{F}(V)}\right.  \left(  a_{\lambda
}\left(  r\right)  \,u_{\lambda}\left\vert _{V}\right.  \right)
=f.$\smallskip\newline In particular, if $r=0$, $u$ and $f$ are called
\emph{associated }in $V$.
\end{definition}

To ensure the independence of the definition with respect to the
representative of $u$, we must have, for any $(\eta_{\lambda})_{\lambda}%
\in\mathcal{J}_{(I_{A},\mathcal{E},\mathcal{P})}(\Omega)$, that $\lim
\limits_{\Lambda}\left.  _{\mathcal{F}(V)}\right.  a_{\lambda}\left(
r\right)  \,\eta_{\lambda}\left\vert _{V}\right.  =0$. As $\mathcal{J}%
_{(I_{A},\mathcal{E},\mathcal{P})}(V)$ is a module over $A$, $(a_{\lambda
}\left(  r\right)  \,\eta_{\lambda}\left\vert _{V}\right.  )_{\lambda}$ is in
$\mathcal{J}_{(I_{A},\mathcal{E},\mathcal{P})}(V)$.\ Thus, our claim follows
from hypothesis (\ref{SPSHypoth1}).

\begin{example}
\label{SPSExample1}Take $X=\mathbb{R}^{d}$, $\mathcal{F}=\mathcal{D}^{\prime}%
$, $\Lambda=]0,1]$, $\mathcal{A=G}$, $V=\Omega$, $r=0$.\ The usual association
between $u=\left[  u_{\varepsilon}\right]  \in\mathcal{G}\left(
\Omega\right)  $ and $T\in\mathcal{D}^{\prime}\left(  \Omega\right)  $ is
defined by%
\[
u\sim T\Longleftrightarrow u\overset{a(0)}{\underset{\mathcal{D}^{\prime
}\left(  \Omega\right)  }{\sim}}T\Longleftrightarrow\lim_{\varepsilon
\rightarrow0}\left.  _{\mathcal{D}^{\prime}\left(  \Omega\right)  }\right.
u_{\varepsilon}=T.
\]

\end{example}

\subsection{The $\mathcal{F}$-singular support of a generalized
function\label{SPSSsecFsingsupp}}

We use the notations of Subsection \ref{SPSSubsecAssoc}. According to the
hypothesis (\ref{SPSHypoth1}), we have, for any open set $\Omega$ in $X$,%
\[
\mathcal{J}_{(I_{A},\mathcal{E},\mathcal{P})}(\Omega)\subset
%\mathcal{N}%
%_{\mathcal{E}}^{\mathcal{F}}\left(  \Omega\right)  =
\left\{  \left(  u_{\lambda}\right)  _{\lambda}\in\mathcal{H}_{(A,\mathcal{E}%
,\mathcal{P})}(\Omega)\ :\ \lim\limits_{\Lambda}\left.  _{\mathcal{F}%
(V)}\right.  u_{\lambda}=0\right\}  .
\]
Set%
\[
\mathcal{F}_{\mathcal{A}}(\Omega)=\left\{  u\in\mathcal{A}(\Omega
)\,\mid\,\exists\left(  u_{\lambda}\right)  _{\lambda}\in u,\ \exists
f\in\mathcal{F}(\Omega)\ :\ \lim\limits_{\Lambda}\left.  _{\mathcal{F}%
(V)}\right.  u_{\lambda}=f\right\}  .
\]
$\mathcal{F}_{\mathcal{A}}(\Omega)$ is well defined because if $\left(
\eta_{\lambda}\right)  _{\lambda}$ belongs to $\mathcal{J}_{(I_{A}%
,\mathcal{E},\mathcal{P})}(\Omega)$, we have $\lim\limits_{\Lambda}\left.
_{\mathcal{F}(V)}\right.  \eta_{\lambda}=0$. \smallskip

Moreover, $\mathcal{F}_{\mathcal{A}}$ is a sub-presheaf of vector spaces
(resp. algebras) of $\mathcal{A}$. Roughly speaking, it is the presheaf whose
sections above some open set $\Omega$ are the generalized functions of
$\mathcal{A}\left(  \Omega\right)  $ associated with an element of
$\mathcal{F}\left(  \Omega\right)  $.

Thus, for $u\in\mathcal{A}\left(  \Omega\right)  $, we can consider the set
$\mathcal{O}_{\mathcal{A}}^{\mathcal{F}}\left(  u\right)  $ of all $x\in
\Omega$ having an open neighborhood $V$ on which $u$ is associated with
$f\in\mathcal{F}\left(  V\right)  $, that is:%
\[
\mathcal{O}_{\mathcal{A}}^{\mathcal{F}}\left(  u\right)  =\left\{  x\in
\Omega\ \left\vert \ \exists V\in\mathcal{V}_{x}:u\left\vert _{V}\right.
\in\mathcal{F}_{\mathcal{A}}(V)\right.  \right\}  ,
\]
$\mathcal{V}_{x}$ being the set of all the open neighborhoods of $x$.
\smallskip

This leads to the following definition:

\begin{definition}
\label{SPSdefSingSupp}The $\mathcal{F}$-singular support of $u\in
\mathcal{A}(\Omega)$\ is denoted $\mathcal{S}_{\mathcal{A}}^{\mathcal{F}%
}\left(  u\right)  $ and defined as%
\[
\mathcal{S}_{\mathcal{A}}^{\mathcal{F}}\left(  u\right)  =\Omega
\backslash\mathcal{O}_{\mathcal{A}}^{\mathcal{F}}\left(  u\right)  .
\]

\end{definition}

\begin{remark}
\label{SPSRemSuppjamad}\ $\left(  i\right)  $~The validity of the gluing
principle $(F_{2})$ is not necessary to get the notion of support (and of
$\mathcal{F}$-singular support) of a section $u\in\mathcal{A}(\Omega)$. More
precisely, the localization principle $(F_{1})$ is sufficient to prove the
following: The set%
\[
\mathcal{O}_{\mathcal{A}}^{\{0\}}\left(  u\right)  =\left\{  x\in
\Omega\ \left\vert \ \exists V\in\mathcal{V}_{x},\ u\left\vert _{V}\right.
=0\right.  \right\}
\]
is exactly the the union $\Omega_{\mathcal{A}}\left(  u\right)  $ of the open
subsets of $\Omega$ on which $u$ vanishes.\newline Indeed, $(F_{1})$ allows to
show that $u$ vanishes on an open subset $\mathcal{O}$ of $\Omega$ if, and
only if, it vanishes on an open neighborhood of every point of $\mathcal{O}$.
This leads immediately to the required assertion.\newline Moreover,
$\Omega_{\mathcal{A}}\left(  u\right)  =\mathcal{O}_{\mathcal{A}}%
^{\{0\}}\left(  u\right)  $ is the largest open set on which $u$ vanishes,
$\mathcal{S}_{\mathcal{A}}^{\{0\}}\left(  u\right)  =\Omega\setminus
\mathcal{O}_{\mathcal{A}}^{\{0\}}\left(  u\right)  $ is exactly the support of
$u$ in its classical definition, and the $\mathcal{F}$-singular support of $u$
is a closed subset of its support.\smallskip\newline$\left(  ii\right)  $~In
contrast to the situation described above for the support, we need the gluing
principle $(F_{2})$ if we want to prove that the restriction of $u$ to
$\mathcal{O}_{\mathcal{A}}^{\mathcal{F}}\left(  u\right)  $ belongs to
$\mathcal{F}_{\mathcal{A}}(\mathcal{O}_{\mathcal{A}}^{\mathcal{F}}\left(
u\right)  )$. We make this precise in the following lemma.
\end{remark}

\begin{lemma}
Take $u\in\mathcal{A}(\Omega)$ and set $\Omega_{\mathcal{A}}^{\mathcal{F}%
}\left(  u\right)  =\cup_{i\in I}\Omega_{i},\left(  \Omega_{i}\right)  _{i\in
I}$ denoting the collection of the open subsets of $\Omega$ such that
$u\left\vert _{\Omega_{i}}\right.  \in\mathcal{F}_{\mathcal{A}}\left(
\Omega_{i}\right)  $. Then, if $\mathcal{F}_{\mathcal{A}}$ is a sheaf (even if
$\mathcal{A}$ is only a prehesaf),\newline$\left(  i\right)  $~$\Omega
_{\mathcal{A}}^{\mathcal{F}}\left(  u\right)  $ is the largest open subset
$\mathcal{O}$ of $\Omega$ such that $u\left\vert _{\mathcal{O}}\right.  $
belongs to $\mathcal{F}_{\mathcal{A}}\left(  \mathcal{O}\right)  $%
;$\newline\left(  ii\right)  ~\Omega_{\mathcal{A}}^{\mathcal{F}}\left(
u\right)  =\mathcal{O}_{\mathcal{A}}^{\mathcal{F}}(u)$ and $\mathcal{S}%
_{\mathcal{A}}^{\mathcal{F}}\left(  u\right)  =\Omega\setminus\Omega
_{\mathcal{A}}^{\mathcal{F}}\left(  u\right)  $.
\end{lemma}

\begin{proof}
$\left(  i\right)  $ For $i\in I$, set $u\left\vert _{\Omega_{i}}\right.
=f_{i}\in\mathcal{F}_{\mathcal{A}}\left(  \Omega_{i}\right)  $.\ The family
$\left(  f_{i}\right)  _{i\in I}$ is coherent by assumption: From $(F_{2})$,
there exists $f\in\mathcal{F}_{\mathcal{A}}(\Omega_{\mathcal{A}}^{\mathcal{F}%
}\left(  u\right)  )$ such that $f\left\vert _{\Omega_{i}}\right.  =f_{i}$.
But from $(F_{1})$, we have $f=u$ on $\cup_{i\in I}\Omega_{i}=\Omega
_{\mathcal{A}}^{\mathcal{F}}\left(  u\right)  $. Thus $u\,|_{\Omega
_{\mathcal{A}}^{\mathcal{F}}\left(  u\right)  }\in\mathcal{F}_{\mathcal{A}%
}(\Omega_{\mathcal{A}}^{\mathcal{F}}\left(  u\right)  )$, and $\Omega
_{\mathcal{A}}^{\mathcal{F}}\left(  u\right)  $ is clearly the largest open
subset of $\Omega$ having this property.

$\left(  ii\right)  $ First, $\mathcal{O}_{\mathcal{A}}^{\mathcal{F}}\left(
u\right)  $ is clearly an open subset of $\Omega$. For $x\in\mathcal{O}%
_{\mathcal{A}}^{\mathcal{F}}\left(  u\right)  $, set $u\left\vert _{V_{x}%
}\right.  =f_{x}\in\mathcal{F}_{\mathcal{A}}\left(  V_{x}\right)  $ for some
suitable neighborhood $V_{x}$.\ The open set $\mathcal{O}_{\mathcal{A}%
}^{\mathcal{F}}\left(  u\right)  $ can be covered by the family $\left(
V_{x}\right)  _{x\in\mathcal{O}_{\mathcal{A}}^{\mathcal{F}}\left(  u\right)
}$.\ As the family $\left(  f_{x}\right)  $ is coherent, we get from $(F_{2})$
that there exists $f\in\mathcal{F}_{\mathcal{A}}\left(  \cup_{x\in
\mathcal{O}_{\mathcal{A}}^{\mathcal{F}}\left(  u\right)  }V_{x}\right)  $ such
that $f\left\vert _{V_{x}}\right.  =f_{x}$. From $(F_{1})$, we have $u=f$ on
$\cup_{x\in\mathcal{O}_{\mathcal{A}}^{\mathcal{F}}\left(  u\right)  }V_{x}$
and, therefore, $u\,|_{\mathcal{O}_{\mathcal{A}}^{\mathcal{F}}\left(
u\right)  }\in\mathcal{F}_{\mathcal{A}}(\mathcal{O}_{\mathcal{A}}%
^{\mathcal{F}}\left(  u\right)  ).$ Thus $\mathcal{O}_{\mathcal{A}%
}^{\mathcal{F}}\left(  u\right)  $ is contained in $\Omega_{\mathcal{A}%
}^{\mathcal{F}}\left(  u\right)  $. Conversely, if $x\in\Omega_{\mathcal{A}%
}^{\mathcal{F}}\left(  u\right)  $, there exists an open neighborhood $V_{x}$
of $x$ such that $u\left\vert _{V_{x}}\right.  \in\mathcal{F}_{\mathcal{A}%
}\left(  V_{x}\right)  $. Thus $x\in\mathcal{O}_{\mathcal{A}}^{\mathcal{F}%
}\left(  u\right)  $ and the assertion $\left(  ii\right)  $ holds.
\end{proof}

\begin{proposition}
\label{SPStheoSingSupp}For any $u,v\in\mathcal{A}(\Omega)$, if $\mathcal{F}$
is a presheaf of topological vector spaces, (resp. algebras), we have:%
\[
\mathcal{S}_{\mathcal{A}}^{\mathcal{F}}(u+v)\subset\mathcal{S}_{\mathcal{A}%
}^{\mathcal{F}}(u)\cup\mathcal{S}_{\mathcal{A}}^{F}(v).
\]
Moreover, in the resp. case, we have%
\[
\mathcal{S}_{\mathcal{A}}^{\mathcal{F}}(uv)\subset\mathcal{S}_{\mathcal{A}%
}^{F}(u)\cup\mathcal{S}_{\mathcal{A}}^{F}(v).
\]

\end{proposition}

\begin{proof}
If $x\in\Omega$ belongs to $\mathcal{O}_{\mathcal{A}}^{\mathcal{F}}%
(u)\cap\mathcal{O}_{\mathcal{A}}^{\mathcal{F}}(v)$, there exist $V$ and $W$ in
$\mathcal{V}_{x}$ such that $u\left\vert _{V}\right.  \in\mathcal{F}%
_{\mathcal{A}}(V)$ and $v\left\vert _{W}\right.  \in\mathcal{F}_{\mathcal{A}%
}(W)$. Thus $(u+v)_{\mid V\cap W}\in\mathcal{F}_{\mathcal{A}}(V\cap W)$ (resp.
$\left(  uv\right)  _{\mid V\cap W}\in\mathcal{F}_{\mathcal{A}}(V\cap W)$),
which implies%
\[
\mathcal{O}_{\mathcal{A}}^{\mathcal{F}}(u)\cap\mathcal{O}_{\mathcal{A}%
}^{\mathcal{F}}(v)\subset\mathcal{O}_{\mathcal{A}}^{\mathcal{F}}%
(u+v)\ \ \ \ \ \text{(resp.\ }\mathcal{O}_{\mathcal{A}}^{\mathcal{F}}%
(u)\cap\mathcal{O}_{\mathcal{A}}^{\mathcal{F}}(v)\subset\mathcal{O}%
_{\mathcal{A}}^{\mathcal{F}}\left(  uv\right)  \text{).}%
\]
The result follows by taking the complementary sets in $\Omega$.\medskip
\end{proof}

This proposition leads easily to the following:

\begin{corollary}
\label{SPStCorSingSupp}Let $\left(  u_{j}\right)  _{1\leq j\leq p}$ be any
finite family of elements in $\mathcal{A}(\Omega)$. If $\mathcal{F}$ is a
presheaf of topological vector spaces, (resp. algebras), we have%
\[
\mathcal{S}_{\mathcal{A}}^{\mathcal{F}}(%
%TCIMACRO{\tsum \limits_{1\leq j\leq p}}%
%BeginExpansion
{\textstyle\sum\limits_{1\leq j\leq p}}
%EndExpansion
u_{j})\subset%
%TCIMACRO{\tbigcup \limits_{1\leq j\leq p}}%
%BeginExpansion
{\textstyle\bigcup\limits_{1\leq j\leq p}}
%EndExpansion
\mathcal{S}_{\mathcal{A}}^{F}(u_{j}).
\]
Moreover, in the resp. case, we have%
\[
\mathcal{S}_{\mathcal{A}}^{\mathcal{F}}(%
%TCIMACRO{\tprod \limits_{1\leq j\leq p}}%
%BeginExpansion
{\textstyle\prod\limits_{1\leq j\leq p}}
%EndExpansion
u_{j})\subset%
%TCIMACRO{\tbigcup \limits_{1\leq j\leq p}}%
%BeginExpansion
{\textstyle\bigcup\limits_{1\leq j\leq p}}
%EndExpansion
\mathcal{S}_{\mathcal{A}}^{F}(u_{j}).
\]
In particular, if $u_{j}=u$ for $1\leq j\leq p$, we have $\mathcal{S}%
_{\mathcal{A}}^{\mathcal{F}}(u^{p})\subset\mathcal{S}_{\mathcal{A}}^{F}(u).$
\end{corollary}

\begin{example}
Taking \label{SPSExampleA}$\mathcal{E}=\mathrm{C}^{\infty}$; $\mathcal{F}%
=\mathcal{D}^{\prime}$; $\mathcal{A}=\mathcal{G}$ leads to the ${\mathcal{D}%
}^{\prime}$-singular support of an element of the Colombeau algebra. This
notion is complementary to the usual concept of local association in the
Colombeau sense. We refer the reader to \cite{JAM0, JAM1} for more details.
\end{example}

\begin{example}
\label{SSPExam1} In the following examples we consider $X=\mathbb{R}^{d}$,
$\mathcal{E}=\mathrm{C}^{\infty}$ and $\mathcal{A=G}$. \smallskip
\newline$\left(  i\right)  $~Take $u\in\sigma_{\Omega}\left(  \mathrm{C}%
^{\infty}\left(  \Omega\right)  \right)  $, where $\sigma_{\Omega}%
:\mathrm{C}^{\infty}\left(  \Omega\right)  \rightarrow\mathcal{G}\left(
\Omega\right)  $ is the canonical embedding defined in Subsection
\ref{SPSSubSecDC}. Then $\mathcal{S}_{\mathcal{G}}^{\mathrm{C}^{p}%
}(u)=\varnothing$, for all $p\in\overline{\mathbb{N}}$.\smallskip
\newline$\left(  ii\right)  $~Take $\varphi\in\mathcal{D}\left(
\mathbb{R}\right)  $, with $\int\varphi\left(  x\right)  \,\mathrm{d}x=1$, and
set $\varphi_{\varepsilon}\left(  x\right)  =\varepsilon^{-1}\varphi\left(
x/\varepsilon\right)  $. As $\varphi_{\varepsilon}\underset{\mathcal{D}%
^{\prime}(\mathbb{R})}{\overset{\varepsilon\rightarrow0}{\longrightarrow}%
}\delta$, we have: $\mathcal{S}_{\mathcal{G}}^{\mathcal{D}^{\prime}}\left(
\left[  \varphi_{\varepsilon}\right]  \right)  =\left\{  0\right\}  $. We note
also that $\mathcal{S}_{\mathcal{G}}^{\mathrm{C}^{p}}\left(  \left[
\varphi_{\varepsilon}\right]  \right)  =\left\{  0\right\}  $. Indeed, for any
$K\Subset\mathbb{R}^{\ast}=\mathbb{R}\backslash\left\{  0\right\}  $ and
$\varepsilon$ small enough, $\varphi_{\varepsilon}$ is null on $K$ and,
therefore, $\varphi_{\varepsilon}\underset{\mathrm{C}^{\infty}(\mathbb{R}%
^{\ast})}{\overset{\varepsilon\rightarrow0}{\longrightarrow}}=0$%
.\smallskip\newline$\left(  iii\right)  $~Take $u=\left[  u_{\varepsilon
}\right]  $ with $u_{\varepsilon}(x)=\varepsilon\sin(x/\varepsilon)$. We have
that $\lim p_{K,0}(u_{\varepsilon})=0$, for all $K\Subset\mathbb{R}$, whereas
$\lim p_{K,1}(u_{\varepsilon})$ does not exist for $l\geq1$.\ Therefore
\[
\mathcal{S}_{\mathcal{G}}^{\mathrm{C}^{0}}(u)=\mathbb{\varnothing
\,},\text{\ \ \ }\mathcal{S}_{\mathcal{G}}^{\mathrm{C}^{1}}(u)=\mathbb{R}%
\text{.}%
\]

\end{example}

\begin{remark}
\label{SPSSingpq} For any $\left(  p,q\right)  \in\overline{\mathbb{N}}^{2}$
with $p\leq q$, and $u\in\mathcal{G}$, it holds that $\mathcal{S}%
_{\mathcal{G}}^{\mathrm{C}^{p}}\left(  u\right)  \subset\mathcal{S}%
_{\mathcal{G}}^{\mathrm{C}^{q}}(u)$.
\end{remark}

\section{The concept of $(a,\mathcal{F})$\textit{-}microlocal
analysis\label{SPSMicAn}}

Let $\Omega$ be an open set in $X$. Fix $u=\left[  u_{\lambda}\right]
\in\mathcal{A}(\Omega)$ and $x\in\Omega$. The idea of the $(a,\mathcal{F)}%
$\textit{-}microlocal analysis is the following: $\left(  u_{\lambda}\right)
_{\lambda}$ may not tend to a section of $\mathcal{F}$ above a neighborhood of
$x$, that is, there exists no $V\in\mathcal{V}_{x}$ and no $f\in
\mathcal{F}\left(  V\right)  $\ such that $\lim\limits_{\Lambda}\left.
_{\mathcal{F}(V)}\right.  u_{\lambda}=f$.\ Nevertheless, in this case, there
may exist $V\in\mathcal{V}_{x}$, $r\geq0$ and $f\in\mathcal{F}\left(
V\right)  $ such that $\lim\limits_{\Lambda}\left.  _{\mathcal{F}(V)}\right.
a_{\lambda}(r)u_{\lambda}=f$, that is $\left[  a_{\lambda}(r)u_{\lambda
}\left\vert _{V}\right.  \right]  $ belongs to the subspace (resp. subalgebra)
$\mathcal{F}_{\mathcal{A}}(V)$ of $\mathcal{A}(V)$ introduced in Subsection
\ref{SPSSsecFsingsupp}.
%\begin{example}
%\label{SPSExam2}Take $u_{\varepsilon}(x)=\dfrac{1}{\varepsilon^{2}}\dfrac
%{x}{1+\varepsilon^{2}x^{2}}$, for $x\in\mathbb{R}$. Then, $u=\left[
%u_{\varepsilon}\right]  $ is a generalized function of $\mathcal{G}%
%(\mathbb{R})$ which is globally and locally not associated with a
%\textrm{C}$^{\infty}$ function. But, for $r\geq2$, $\left[  \varepsilon
%^{r}u_{\varepsilon}\right]  $ is globally associated with a \textrm{C}$^{\infty
%}$ function.
%\end{example}
These preliminary remarks lead to the following concept.

\subsection{The $(a,\mathcal{F})$\textit{-}singular parametric
spectrum\label{SPSSubSecSPS}}

We recall that $a$ is a map from $\mathbb{R}_{+}$ to $A_{+}$ such that $a(0) =
1$ and $\mathcal{F}$ is a presheaf of topological vector spaces (or
topological algebras). For any open subset $\Omega$ of $X$, $u=\left[
u_{\lambda}\right]  \in\mathcal{A}(\Omega)$ and $x\in\Omega,$ set
\begin{align*}
N_{\left(  a,\mathcal{F}\right)  ,x}\left(  u\right)   &  =\left\{
r\in\mathbb{R}_{+}\ \mid\ \exists V\in\mathcal{V}_{x},\ \exists f\in
\mathcal{F}(V)\ :\ \lim\limits_{\Lambda}\left.  _{\mathcal{F}(V)}\right.
(a_{\lambda}(r)\,u_{\lambda}\left\vert _{V}\right.  )=f\right\} \\
&  =\Big\{ r\in\mathbb{R}_{+}\ \mid\ \exists V\in\mathcal{V}_{x}\ :\ \left[
a_{\lambda}\left(  r\right)  u_{\lambda}\left\vert _{V}\right.  \right]
\in\mathcal{F}_{\mathcal{A}}(V)\Big\} .
\end{align*}
It is easy to check that $N_{\left(  a,\mathcal{F}\right)  ,x}\left(
u\right)  $ does not depend on the representative of $u$. If no confusion may
arise, we shall simply write
\[
N_{\left(  a,\mathcal{F}\right)  ,x}\left(  u\right)  =N_{x}(u).
\]

\begin{theorem}
\label{SPSThmNx}Suppose that:\newline$\left(  a\right)  $ For all $\lambda
\in\Lambda$
\[
\forall\left(  r,s\right)  \in\mathbb{R}_{+},\ \ a_{\lambda}(r+s)\leq
a_{\lambda}(r)a_{\lambda}(s),
\]
and, for all $r\in\mathbb{R}_{+}\backslash\left\{  0\right\}  $, the net
$\left(  a_{\lambda}\left(  r\right)  \right)  _{\lambda}$ converges to $0$ in
$\mathbb{K}$.\newline$\left(  b\right)  $ $\mathcal{F}$ is a presheaf of
separated locally convex topological vector spaces.\smallskip\newline Then we
have, for $u\in\mathcal{A}(\Omega)$:\newline$\left(  i\right)  $~If $r\in
N_{x}(u)$, then $\left[  r,+\infty\right)  $ is included in $N_{x}(u)$.
Moreover, for all $s>r$, there exists $V\in\mathcal{V}_{x}$ such that:
$\lim\limits_{\Lambda}\left.  _{\mathcal{F}(V)}\right.  (a_{\lambda
}(s)\,u_{\lambda}\left\vert _{V}\right.  )=0$. Consequently, $N_{x}(u)$ is
either empty, or a sub-interval of $\mathbb{R}_{+}$.\newline$\left(
ii\right)  $~More precisely, suppose that for $x\in\Omega$, there exist
$r\in\mathbb{R}_{+}$, $V\in\mathcal{V}_{x}$ and$\ f\in\mathcal{F}(V)$, nonzero
on each neighborhood of $x$ included in $V$, such that $\lim\limits_{\Lambda
}\left.  _{\mathcal{F}(V)}\right.  (a_{\lambda}(r)\,u_{\lambda}\left\vert
_{V}\right.  )=f$. Then $N_{x}(u)=\left[  r,+\infty\right)  .\newline\left(
iii\right)  $~In the situation of $\left(  i\right)  $ and $\left(  ii\right)
$, we have that $0\in N_{x}(u)$ iff $N_{x}(u)=\mathbb{R}_{+}$.\ Moreover, if
one of these assertions holds, the limits $\lim\limits_{\Lambda}\left.
_{\mathcal{F}(V)}\right.  (a_{\lambda}\left(  s\right)  \,u_{\lambda
}\left\vert _{V}\right.  )$ can be non null only for $s=0$.
\end{theorem}

\begin{proof}
$\left(  i\right)  $~If $r\in N_{x}(u)$, there exist $V\in\mathcal{V}_{x}$ and
$f\in\mathcal{F(}V)$ such that $\lim\limits_{\Lambda}\left.  _{\mathcal{F}%
(V)}\right.  (a_{\lambda}(r)u_{\lambda\mid_{V}})=f.$ As $\mathcal{F}(V)$ is
locally convex, its topology may be described by a family $\mathcal{Q}%
_{V}=\left(  q_{j}\right)  _{j\in J\left(  V\right)  }$ of semi-norms. For all
$s>r$, we have, for any $j\in J\left(  V\right)  $,
\[
q_{j}(a_{\lambda}(s)\,(\,u_{\lambda}\left\vert _{V}\right.  ))=a_{\lambda
}(s)\,q_{j}(\,u_{\lambda}\left\vert _{V}\right.  )\leq a_{\lambda
}(s-r)\,a_{\lambda}(r)\,q_{j}(\,u_{\lambda}\left\vert _{V}\right.  )\leq
a_{\lambda}(s-r)\,q_{j}(a_{\lambda}(r)\,\,u_{\lambda}\left\vert _{V}\right.
).
\]
From $\lim\limits_{\Lambda}q_{j}\left(  a_{\lambda}(r)\,(u_{\lambda}\left\vert
_{V}\right.  -f)\right)  =0$, we have $q_{j}(a_{\lambda}(r)\,\,u_{\lambda
}\left\vert _{V}\right.  )<+\infty$ and $\lim\limits_{\Lambda}q_{j}\left(
a_{\lambda}(s)(\,u_{\lambda}\left\vert _{V}\right.  )\right)  =0$, since
$a_{\lambda}(s-r)\overset{\Lambda}{\rightarrow}0$. Thus $\lim\limits_{\Lambda
}\left.  _{\mathcal{F}(V)}\right.  (a_{\lambda}(s)\,u_{\lambda}\left\vert
_{V}\right.  )=0.$\smallskip

\noindent$\left(  ii\right)  $~From $\left(  i\right)  $, we have $\left[
r,+\infty\right)  \subset N_{x}(u)$. Suppose that there exists $t<r$ in
$N_{x}(u)$. Then we get $W\in\mathcal{V}_{x}$, which can be chosen included in
$V$, and $g\in\mathcal{F(}W)$ such that $\lim\limits_{\Lambda}\left.
_{\mathcal{F}(W)}\right.  (a_{\lambda}(t)\,u_{\lambda}\left\vert _{W}\right.
)=g$.\ With the notations of the proof of $\left(  i\right)  $, we have%
\[
q_{j}(a_{\lambda}(r)\,(u_{\lambda}\left\vert _{W}\right.  ))\leq a_{\lambda
}(r-t)q_{j}(a_{\lambda}(t)u_{\lambda}\left\vert _{W}\right.  ).
\]
As $q_{j}(a_{\lambda}(t)u_{\lambda}\left\vert _{V}\right.  )$ is bounded, it
follows that $\lim\limits_{\Lambda}q_{j}(a_{\lambda}(r)\,(u_{\lambda
}\left\vert _{W}\right.  ))=0$, which is in contradiction with $\lim
\limits_{\Lambda}\left.  _{\mathcal{F}(V)}\right.  (a_{\lambda}%
(r)\,(u_{\lambda}\left\vert _{V}\right.  )=f\not \equiv 0$ on $W.$\smallskip

\noindent$\left(  iii\right)  ~$The first assertion follows directly from
$\left(  i\right)  $ and the second from $\left(  ii\right)  $.\medskip
\end{proof}

From now on, we suppose that the hypotheses $(a)$ and $(b)$ of Theorem
\ref{SPSThmNx} are fulfilled. We set
\begin{gather*}
\Sigma_{\left(  a,\mathcal{F}\right)  ,x}(u)=\Sigma_{x}(u)=\mathbb{R}%
_{+}\backslash N_{x}(u),\\
R_{\left(  a,\mathcal{F}\right)  ,x}\left(  u\right)  =R_{x}(u)=\inf N_{x}(u).
\end{gather*}
According to the previous remarks and comments, $\Sigma_{\left(
a,\mathcal{F}\right)  ,x}(u)$ is an interval of $\mathbb{R}_{+}$ of the form
$\left[  0,R_{\left(  a,\mathcal{F}\right)  ,x}\left(  u\right)  \right)  $ or
$\left[  0,R_{\left(  a,\mathcal{F}\right)  ,x}\left(  u\right)  \right]  $,
the empty set, or $\mathbb{R}_{+}$.

\begin{definition}
\label{SPSDefSingSp}\textit{T}he $\left(  a,\mathcal{F}\right)  $%
\emph{-singular spectrum} of $u\in\mathcal{A}(\Omega)$ is the set%
\[
\mathcal{S}_{\mathcal{A}}^{\left(  a,\mathcal{F}\right)  }\left(  u\right)
=\left\{  (x,r)\in\Omega\times\mathbb{R}_{+}\,\left\vert \,r\in\Sigma
_{x}(u)\right.  \right\}  .
\]

\end{definition}

\begin{example}
\label{SPSExam3}Take $X=\mathbb{R}^{d}$, $\mathcal{E}=\mathrm{C}^{\infty}$,
$\mathcal{F}=\mathrm{C}^{p}$ ($p\in\overline{\mathbb{N}}=\mathbb{N\cup
}\left\{  +\infty\right\}  $), $f\in\mathrm{C}^{\infty}\left(  \Omega\right)
$.\ Set $u=\left[  \left(  \varepsilon^{-1}f\right)  _{\varepsilon}\right]  $
and $v=\left[  \left(  \varepsilon^{-1}\left\vert \ln\varepsilon\right\vert
f\right)  _{\varepsilon}\right]  $ in $\mathcal{A}\left(  \Omega\right)
=\mathcal{G}\left(  \Omega\right)  $. Then, for all $x\in\mathbb{R}$,
\[
N_{\left(  a,\mathrm{C}^{p}\right)  ,x}\left(  u\right)  =\left[
1,+\infty\right)  \,,\ \ \text{ }N_{\left(  a,\mathrm{C}^{p}\right)
,x}\left(  v\right)  =\left(  1,+\infty\right)  \,,\ \ \ R_{\left(
a,\mathrm{C}^{p}\right)  ,x}\left(  u\right)  =R_{\left(  a,\mathrm{C}%
^{p}\right)  ,x}\left(  v\right)  =1.
\]

\end{example}

\begin{remark}
\label{SPSRemSPS}We have: $\Sigma_{\left(  a,\mathcal{F}\right)
,x}(u)=\varnothing$ iff $N_{\left(  a,\mathcal{F}\right)  ,x}(u)=\mathbb{R}%
_{+}$ and, according to Theorem \ref{SPSThmNx}, iff $0\in N_{\left(
a,\mathcal{F}\right)  ,x}(u)$, that is, there exist $\left(  V,f\right)
\in\mathcal{V}_{x}\times\mathcal{F}(V)$ such that $\lim\limits_{\Lambda
}\left.  _{\mathcal{F}(V)}\right.  (a_{\lambda}(0)\,u_{\lambda}\left\vert
_{V}\right.  )=f$.\ As $a_{\lambda}(0)\equiv1$, this last assertion is
equivalent to $x\in\mathcal{O}_{\mathcal{A}}^{\mathcal{F}}\left(  u\right)  $.
Thus $\Sigma_{\left(  a,\mathcal{F}\right)  ,x}(u)=\varnothing$ iff
$x\notin\mathcal{S}_{\mathcal{A}}^{\mathcal{F}}(u)$.
\end{remark}

This remark implies directly the:

\begin{proposition}
\label{SPSProjSPS}The projection of the $\left(  a,\mathcal{F}\right)
$\emph{-}singular spectrum of $u$ on $\Omega$ is the $\mathcal{F}$-singular
support of $u$.
\end{proposition}

\subsection{Example: The Colombeau case}

In this subsection we investigate the relationship between the $\left(
a,\mathcal{F}\right)  $\emph{-}singular spectrum and the sharp topology for
$X=\mathbb{R}^{d}$, $\mathcal{E}=\mathrm{C}^{\infty}$, $\mathcal{F}%
=\mathrm{C}^{p}$ ($p\in\mathbb{N}$), $\mathcal{A}=\mathcal{G}$,
$a_{\varepsilon}\left(  r\right)  =\varepsilon^{r}$.\ First, let us remark
that, for $u=\left[  u_{\varepsilon}\right]  \in\mathcal{G}\left(
\Omega\right)  $, $x\in\Omega$ $(\Omega\in\mathcal{O}\left(  \mathbb{R}%
^{d}\right)  $), $N_{\left(  a,\mathrm{C}^{p}\right)  ,x}\left(  u\right)  $
is never empty.

Indeed, consider $V\in\mathcal{V}_{x}$ with $\overline{V}\Subset\Omega$. There
exists $m>0$ such that $p_{p,\overline{V}}\left(  u_{\varepsilon}\right)
=\mathrm{o}\left(  \varepsilon^{-m}\right)  $ as $\varepsilon\rightarrow
0$.\ Thus, $p_{k,\overline{V}}\left(  u_{\varepsilon}\right)  =o\left(
\varepsilon^{-m}\right)  $ for all $k\leq p$ and $\lim\limits_{\varepsilon
\rightarrow0}\left.  _{\mathrm{C}^{p}(V)}\right.  \left(  \varepsilon
^{m}u_{\varepsilon}\left\vert _{V}\right.  \right)  =0$.$\ $Thus $\left[
m,+\infty\right)  \subset N_{\left(  a,\mathrm{C}^{p}\right)  ,x}\left(
u\right)  .\smallskip$

Let us now recall the construction of the sharp topology on $\mathcal{G}%
\left(  \Omega\right)  $ . For $u=\left[  \left(  u_{\varepsilon}\right)
_{\varepsilon}\right]  \in\mathcal{G}\left(  \Omega\right)  $, $K\Subset
\Omega$, $l\in\mathbb{N}$, set%
\[
v_{K,l}(u)=\inf\left\{  r\in\mathbb{R\,}\left\vert \mathbb{\,}p_{K,l}\left(
u_{\varepsilon}\right)  =\mathrm{o}(\varepsilon^{-r})\ \text{as }%
\varepsilon\rightarrow0\right.  \right\}
\]
The real number $v_{K,l}(u)$ is well defined, i.e. does not depend on the
representative of $u$, and is called the $\left(  K,l\right)  $%
-\emph{valuation }of $u$. It has the usual properties:\newline$\left(
i\right)  ~\forall\lambda\in\mathbb{C}\backslash\{0\},\;\forall u\in
\mathcal{G}\left(  \Omega\right)  ,\;$ $v_{K,l}(\lambda u)=v_{K,l}(u)$
;\newline$\left(  ii\right)  ~\forall u,v\in\mathcal{G}\left(  \Omega\right)
,\;$ $v_{K,l}(u+v)\leq\sup(v_{K,l}(u),v_{K,l}(v))$.$\smallskip$

The family $\left(  v_{K,l}\right)  $ permits to define the $\left(
K,l\right)  $\emph{-pseudodistances} $d_{K,l}$ on $\mathcal{G}\left(
\Omega\right)  $ by
\[
\forall\left(  u,v\right)  \in\mathcal{G}\left(  \Omega\right)  ^{2}%
,\ \ \ d_{K,l}\left(  u,v\right)  =\exp\left(  v_{K,l}(u-v)\right)  ,
\]
which turns out to be ultrametric:%
\[
\forall\left(  u,v,w\right)  \in\mathcal{G}\left(  \Omega\right)
^{3},\ \ \ d_{K,l}\left(  u,v\right)  \leq\sup(d_{K,l}(u,w),d_{K,l}(w,v)).
\]
The topology defined by the family $\left(  d_{K,l}\right)  _{K,l}$ is called
the sharp topology on $\mathcal{G}\left(  \Omega\right)  $.\smallskip

As we are interested here in valuations greater or equal to $0$, we set, for
$u\in\mathcal{G}\left(  \Omega\right)  $,
\[
\nu_{K,l}(u)=\sup\left(  v_{K,l}(u),0\right)  .
\]
We can define, for $x\in\Omega$, the $l$-\emph{valuation }of $u$ at $x$ by
\[
\nu_{x,l}(u)=\inf\left\{  \nu_{\overline{V},l}(u)\mathbb{\,}\left\vert
\mathbb{\,}V\in\mathcal{V}\left(  x\right)  ,\ V\text{ relatively
compact}\right.  \right\}
\]
and set, for any $p\in\overline{\mathbb{N}}$,
\[
\nu_{x}^{p}(u)=\sup_{0\leq l\leq p}\nu_{x,l}(u).
\]

\begin{proposition}
\label{SPSSharpAF}For all $p\in\overline{\mathbb{N}}$, $\left[  u_{\varepsilon
}\right]  \in\mathcal{G}\left(  \Omega\right)  $ and $x\in\Omega$, we have
\[
\nu_{x}^{p}(u)=R_{\left(  a,\mathrm{C}^{p}\right)  ,x}\left(  u\right)  =\inf
N_{\left(  a,\mathrm{C}^{p}\right)  ,x}\left(  u\right)  .
\]

\end{proposition}

\begin{proof}
Take $r>\nu_{x}^{p}(u)$.\ Then, for any $l$ with $0\leq l\leq p$, one has
$r>\nu_{x,l}(u)$ and there exists $V\in\mathcal{V}\left(  x\right)  ,\ V$
relatively compact, such that $v_{\overline{V},l}(u)<r$. Thus, $\mathbb{\,}%
p_{\overline{V},l}\left(  u_{\varepsilon}\right)  =\mathrm{o}(\varepsilon
^{-r}),\ $as $\varepsilon\rightarrow0$, and $\lim\limits_{\varepsilon
\rightarrow0}\left.  _{\mathrm{C}^{p}(V)}\right.  (\varepsilon^{r}%
u_{\varepsilon}\left\vert _{V}\right.  )=0$, which implies that $r>R_{\left(
a,\mathrm{C}^{p}\right)  ,x}\left(  u\right)  $ and $\nu_{x}^{p}(u)\geq$
$R_{\left(  a,\mathrm{C}^{p}\right)  ,x}\left(  u\right)  $. Conversely, if
$r>R_{\left(  a,\mathrm{C}^{p}\right)  ,x}\left(  u\right)  $, there exists
$V\in\mathcal{V}\left(  x\right)  $ such that $\lim\limits_{\varepsilon
\rightarrow0}\left.  _{\mathrm{C}^{p}(V)}\right.  (\varepsilon^{r}%
u_{\varepsilon}\left\vert _{V}\right.  )=0$.\ For any relatively compact
neighborhood $W$ of $x$ included in $V$, we get $p_{\overline{W},l}\left(
u_{\varepsilon}\right)  =\mathrm{o}(\varepsilon^{-r})$ and $r>v_{\overline
{W},l}(u)>\nu_{x,l}(u)$. Thus, $r\geq\nu_{x}^{p}(u)$ and $\nu_{x}^{p}(u)\leq
R_{\left(  a,\mathrm{C}^{p}\right)  ,x}\left(  u\right)  $.
\end{proof}

\subsection{Some properties of the $(a,\mathcal{F})$-singular parametric
spectrum}

\begin{notation}
For $u=\left[  u_{\lambda}\right]  \in\mathcal{A}\left(  \Omega\right)  $,
$\lim\limits_{\Lambda}\left.  _{\mathcal{F}(V)}\right.  \left(  a_{\lambda
}(r)\,u_{\lambda}\left\vert _{V}\right.  \right)  \in\mathcal{F}\left(
V\right)  $ means that there exists $f\in\mathcal{F}\left(  V\right)  $ such
that $\lim\limits_{\Lambda}\left.  _{\mathcal{F}(V)}\right.  \left(
a_{\lambda}(r)\,u_{\lambda}\left\vert _{V}\right.  \right)  =f$.
\end{notation}

\subsubsection{Linear properties}

\begin{proposition}
\label{SPSSpSsum}For any $u,v\in\mathcal{A}(\Omega)$, we have%
\[
\mathcal{S}_{\mathcal{A}}^{\left(  a,\mathcal{F}\right)  }\left(  u+v\right)
\subset\mathcal{S}_{\mathcal{A}}^{\left(  a,\mathcal{F}\right)  }\left(
u\right)  \cup\mathcal{S}_{\mathcal{A}}^{\left(  a,\mathcal{F}\right)
}\left(  v\right)  .
\]

\end{proposition}

\begin{proof}
Let $r$ be in $N_{x}(u)\cap N_{x}(v)$. Then there exist $V\in\mathcal{V}_{x}$
and $W\in\mathcal{V}_{x}$ such that%
\[
\lim\limits_{\Lambda}\left.  _{\mathcal{F}(V)}\right.  \left(  a_{\lambda
}(r)\,u_{\lambda}\left\vert _{V}\right.  \right)  \in\mathcal{F}\left(
V\right)  \text{ and }\lim\limits_{\Lambda}\left.  _{\mathcal{F}(W)}\right.
\left(  a_{\lambda}(r)\,v_{\lambda}\left\vert _{W}\right.  \right)
\in\mathcal{F}\left(  W\right)  \text{.}%
\]
Thus $\lim\limits_{\Lambda}\left.  _{\mathcal{F}(V\cap W)}\right.
(a_{\lambda}(r)\left(  u_{\lambda}+v_{\lambda}\right)  \left\vert _{V\cap
W}\right.  )\in\mathcal{F}\left(  V\cap W\right)  $ and $r\in N_{x}(u+v)$.
Consequently,
\[
N_{x}(u)\cap N_{x}(v)\subset N_{x}(u+v)\text{.}%
\]
We obtain the result by taking the complementary sets in $\mathbb{R}_{+}$.
\end{proof}

\begin{corollary}
\label{SPSSpScor1}\textit{F}or any $u$, $u_{0}$, $u_{1}$\ in $\mathcal{A}%
(\Omega)$ with%
\[
\left(  i\right)  \mathit{\ }u=u_{0}+u_{1}\ \ \ \ \ \ (ii)\mathit{\ }%
\mathcal{S}_{\mathcal{A}}^{\left(  a,\mathcal{F}\right)  }\left(
u_{0}\right)  =\varnothing,
\]
\textit{we have}%
\[
\mathcal{S}_{\mathcal{A}}^{\left(  a,\mathcal{F}\right)  }\left(  u\right)
=\mathcal{S}_{\mathcal{A}}^{\left(  a,\mathcal{F}\right)  }\left(
u_{1}\right)  \text{.}%
\]

\end{corollary}

\begin{proof}
Proposition \ref{SPSSpSsum} and condition $(ii)$ give $\mathcal{S}%
_{\mathcal{A}}^{\left(  a,\mathcal{F}\right)  }\left(  u\right)
\subset\mathcal{S}_{\mathcal{A}}^{\left(  a,\mathcal{F}\right)  }\left(
u_{1}\right)  $. As $(i)$ implies $u_{0}=u-u_{1}$, we obtain the converse
inclusion, and thus the equality.
\end{proof}

\subsubsection{Differential properties}

We suppose that $\mathcal{F}$ is a sheaf of topological differential vector
spaces (resp. algebras), with continuous differentiation, admitting
$\mathcal{E}$ as a subsheaf of topological differential algebras. Then the
sheaf $\mathcal{A}$ is also a sheaf of differential algebras with, for any
$\alpha\in\mathbb{N}^{d}$ and $u\in\mathcal{A}\left(  \Omega\right)  $,
\[
\partial^{\alpha}u=\left[  \partial^{\alpha}u_{\lambda}\right]  \text{, where
}\left(  u_{\lambda}\right)  _{\lambda}\text{ is any representative of
}u\text{.}%
\]
The independence of $\partial^{\alpha}u$ on the choice of representative
follows directly from the definition of $\mathcal{J}_{(I_{A},\mathcal{E}%
,\mathcal{P})}$.)

\begin{proposition}
\label{SPSSpSDerivee}\textit{Let }$u\ $be in $\mathcal{A}(\Omega)$\textit{.
For all }$\partial^{\alpha}$, $\alpha\in\mathbb{N}^{d}$, we have\textit{\ }%
\[
\mathcal{S}_{\mathcal{A}}^{\left(  a,\mathcal{F}\right)  }\left(
\partial^{\alpha}u\right)  \subset\mathcal{S}_{\mathcal{A}}^{\left(
a,\mathcal{F}\right)  }\left(  u\right)  \text{.}%
\]

\end{proposition}

\begin{proof}
Take $u\in\mathcal{A}(\Omega)$, $\alpha\in\mathbb{N}^{d}$, $x\in\Omega$, $r\in
N_{x}(u)$. There exists $V\in\mathcal{V}_{x},$ $f\in\mathcal{F}\left(
V\right)  $ such that%
\[
\lim\limits_{\Lambda}\left.  _{\mathcal{F}(V)}\right.  \left(  a_{\lambda
}(r)\,u_{\lambda}\left\vert _{V}\right.  \right)  =f.
\]
The continuity of $\partial^{\alpha}$ implies that
\[
\lim\limits_{\Lambda}\left.  _{\mathcal{F}(V)}\right.  \left(  a_{\lambda
}(r)\partial^{\alpha}\,u_{\lambda}\left\vert _{V}\right.  \right)
=\partial^{\alpha}f.
\]
Thus $N_{x}(u)\subset N_{x}(\partial^{\alpha}u)$. The result is
proved.\medskip
\end{proof}

In the following two results we require that $\mathcal{F}$ is a sheaf of
topological modules over $\mathcal{E}$, in addition. The proofs are straightforward.

\begin{proposition}
\label{SPSSpSMult}\textit{Let }$g$ be in $\mathcal{E}(\Omega)$ and $u\ $in
$\mathcal{A}(\Omega)$\textit{. We have}%
\[
\mathcal{S}_{\mathcal{A}}^{\left(  a,\mathcal{F}\right)  }\left(  gu\right)
\subset\mathcal{S}_{\mathcal{A}}^{\left(  a,\mathcal{F}\right)  }\left(
u\right)  \text{.}%
\]

\end{proposition}

Propositions \ref{SPSSpSsum}, \ref{SPSSpSDerivee} and \ref{SPSSpSMult} finally imply:

\begin{corollary}
\label{SPSCinfOPD}\textit{Let }$P(\partial)=%
%TCIMACRO{\dsum \limits_{\left\vert \alpha\right\vert \leq m}}%
%BeginExpansion
{\displaystyle\sum\limits_{\left\vert \alpha\right\vert \leq m}}
%EndExpansion
C_{\alpha}\partial^{\alpha}$\textit{\ be a differential polynomial with
coefficients in }$\mathcal{E}(\Omega).$ \textit{For any }$u\in\mathcal{A}%
(\Omega)$\textit{, we have}%
\[
\mathcal{S}_{\mathcal{A}}^{\left(  a,\mathcal{F}\right)  }\left(
P(\partial)u\right)  \subset\mathcal{S}_{\mathcal{A}}^{\left(  a,\mathcal{F}%
\right)  }\left(  u\right)  \text{.}%
\]

\end{corollary}

\subsubsection{Nonlinear properties}

\begin{theorem}
\label{SPSAFSASProd}For given $u$ and $v\in\mathcal{A}(\Omega)$, let $D_{i}$
($i=1,2,3$) be the following disjoint sets:%
\[
D_{1}=\mathcal{S}_{\mathcal{A}}^{\mathcal{F}}(u)\diagdown(\mathcal{S}%
_{\mathcal{A}}^{\mathcal{F}}(u)\cap\mathcal{S}_{\mathcal{A}}^{\mathcal{F}%
}(v))\ ;\ \ D_{2}=\mathcal{S}_{\mathcal{A}}^{\mathcal{F}}(v)\diagdown
(\mathcal{S}_{\mathcal{A}}^{\mathcal{F}}(u)\cap\mathcal{S}_{\mathcal{A}%
}^{\mathcal{F}}(v))\ ;\ \ D_{3}=\mathcal{S}_{\mathcal{A}}^{\mathcal{F}}%
(u)\cap\mathcal{S}_{\mathcal{A}}^{\mathcal{F}}(v).
\]
Then the ($a$,$\mathcal{F})$-singular asymptotic spectrum of $uv$ verifies%
\[
\mathcal{S}_{\mathcal{A}}^{\left(  a,\mathcal{F}\right)  }\left(  uv\right)
\subset\left\{  (x,\Sigma_{x}(u)),x\in D_{1}\right\}  \cup\left\{
(x,\Sigma_{x}(v)),x\in D_{2}\right\}  \cup\left\{  (x,E_{x}(u,v)),x\in
D_{3}\right\}
\]
where (for any $x\in D_{3}$)%
\[
E_{x}(u,v)=\left\{
\begin{array}
[c]{l}%
\lbrack0,\sup\Sigma_{x}(u)+\sup\Sigma_{x}(v)]\text{ if }\Sigma_{x}%
(u)\neq\mathbb{R}_{+}\text{ and }\Sigma_{x}(v)\neq\mathbb{R}_{+}\\
\mathbb{R}_{+}\text{ if }\Sigma_{x}(u)=\mathbb{R}_{+}\text{ or }\Sigma
_{x}(v)=\mathbb{R}_{+}%
\end{array}
\right.
\]

\end{theorem}

\begin{proof}
Suppose that $x$ belongs to $D_{1}$. Then $x$ is not in $\mathcal{S}%
_{\mathcal{A}}^{\mathcal{F}}(v)$ and we have%
\[
\Sigma_{x}(v)=\varnothing\text{, }N_{x}(v)=\mathbb{R}_{+}.
\]
If $N_{x}(u)$ is not empty, let $r$ be in $N_{x}(u)$. As $N_{x}(v)=\mathbb{R}%
_{+}$, we have $r\in N_{x}(v)$. Thus there exists $V\in\mathcal{V}_{x}$ (resp.
$W\in\mathcal{V}_{x}$) such that $\left[  a_{\lambda}(r)u_{\lambda}\left\vert
_{V}\right.  \right]  \in\mathcal{F}_{\mathcal{A}}(V)$ (resp. $\left[
a_{\lambda}(r)v_{\lambda}\left\vert _{W}\right.  \right]  \in\mathcal{F}%
_{\mathcal{A}}(W)$). As $\mathcal{F}$ is a sheaf of topological algebras we
have%
\[
\left[  a_{\lambda}(r)\,\left(  u_{\lambda}v_{\lambda}\right)  \left\vert
_{V\cap W}\right.  \right]  \in\mathcal{F}_{\mathcal{A}}(V\cap W).
\]
Thus, $r$ belongs to $N_{x}(uv)$. Therefore, we have proved that $\Sigma
_{x}(uv)\subset\Sigma_{x}(u)$. If $N_{x}(u)$ is empty, we have $\Sigma
_{x}(u)=\mathbb{R}_{+}$ and the above inclusion is obviously fulfilled. For
$x$ in $D_{2}$, the same proof gives $\Sigma_{x}(uv)\subset\Sigma
_{x}(v).\smallskip$

\noindent Consider $x$ in $D_{3}$.\ Then, $\Sigma_{x}(u)$ and $\Sigma_{x}(v)$
are not empty. We suppose first that both of them are not equal to
$\mathbb{R}_{+}$. Set $R=\sup\Sigma_{x}(u)$ and $S=\sup\Sigma_{x}(v)$. If
$r>R$, there exists $r^{\prime}\in N_{x}(u)$ such that $R<r^{\prime}<r$ and
then, from the part $(i)$ of Theorem \ref{SPSThmNx}, there exists
$V\in\mathcal{V}_{x}$ such that
\[
\lim\limits_{\Lambda}\left.  _{\mathcal{F}(V)}\right.  \left(  a_{\lambda
}(r)\,u_{\lambda}\left\vert _{V}\right.  \right)  =0.
\]
Similarly, if $s>S$, there exists $W\in\mathcal{V}_{x}$ such that%
\[
\lim\limits_{\Lambda}\left.  _{\mathcal{F}(W)}\right.  \left(  a_{\lambda
}(s)\,v_{\lambda}\left\vert _{W}\right.  \right)  =0.
\]
Then $\lim\limits_{\Lambda}\left.  _{\mathcal{F}(V\cap W)}\right.  \left(
a_{\lambda}(r)a_{\lambda}(s)\,\left(  u_{\lambda}v_{\lambda}\right)
{}\left\vert _{V\cap W}\right.  \right)  =0$. By expressing this limit in
terms of semi-norms, as in the proof of Theorem \ref{SPSThmNx} and by using
the inequality $a_{\lambda}(r+s)\leq a_{\lambda}(r)a_{\lambda}(s)$, we get
that $\lim\limits_{\Lambda}\left.  _{\mathcal{F}(V\cap W)}\right.  \left(
a_{\lambda}(r+s)\,\left(  u_{\lambda}v_{\lambda}\right)  \left\vert _{V\cap
W}\right.  \right)  =0$.\ Thus
\[
\left[  r+s,\infty\right[  \subset N_{x}(uv)\text{\ or }\left[  0,r+s\right[
\supset\Sigma_{x}(uv)
\]
for any $r>R$ and $s>S$. Thus%
\[
\Sigma_{x}(uv)\subset\left[  0,R+S\right]  =[0,\sup\Sigma_{x}(u)+\sup
\Sigma_{x}(v)].
\]
If $\Sigma_{x}(u)$ or $\Sigma_{x}(v)$ is equal to $\mathbb{R}_{+}$, the
obvious inclusion $\Sigma_{x}(uv)\subset\mathbb{R}_{+}$ gives the last result.
\end{proof}

\begin{corollary}
\label{SPSAFSASPow}For given $u$ $\in\mathcal{A}(\Omega)$ and $p\in
\mathbb{N}^{\ast}$, we have%
\[
\mathcal{S}_{\mathcal{A}}^{\left(  a,\mathcal{F}\right)  }\left(
u^{p}\right)  \subset\left\{  (x,H_{p,x}(u)),x\in\mathcal{S}_{\mathcal{A}%
}^{\mathcal{F}}(u)\right\}  .
\]
where\ $H_{p,x}(u)=\left\{
\begin{array}
[c]{l}%
\lbrack0,p\sup\Sigma_{x}(u)]\text{ if }\Sigma_{x}(u)\neq\mathbb{R}_{+}\\
\mathbb{R}_{+}\text{ if }\Sigma_{x}(u)=\mathbb{R}_{+}%
\end{array}
\right.  $
\end{corollary}

\begin{proof}
When $\Sigma_{x}(u)=\mathbb{R}_{+}$, the result is obvious. Suppose now
$\Sigma_{x}(u)\neq\mathbb{R}_{+}$. We shall prove the result by induction. If
$p=1$, the result is a simple consequence of the definitions. Suppose that the
result holds for some $p\geq1$. Set $v=u^{p}$ in the previous theorem. We have%
\[
D_{1}=\mathcal{S}_{\mathcal{A}}^{\mathcal{F}}(u)\diagdown\mathcal{S}%
_{\mathcal{A}}^{\mathcal{F}}(u^{p})\ ;\ \ D_{2}=\varnothing\ ;\ \ D_{3}%
=\mathcal{S}_{\mathcal{A}}^{\mathcal{F}}(u^{p}).
\]
Thus%
\[
\mathcal{S}_{\mathcal{A}}^{\left(  a,\mathcal{F}\right)  }\left(
u^{p+1}\right)  \subset\left\{  (x,\Sigma_{x}(u)),x\in\mathcal{S}%
_{\mathcal{A}}^{\mathcal{F}}(u)\diagdown\mathcal{S}_{\mathcal{A}}%
^{\mathcal{F}}(u^{p})\right\}  \cup\left\{  (x,[0,(p+1)\sup\Sigma
_{x}(u)]),x\in\mathcal{S}_{\mathcal{A}}^{\mathcal{F}}(u^{p})\right\}  ,
\]
by using the induction hypothesis. It follows a fortiori that
\[
\mathcal{S}_{\mathcal{A}}^{\left(  a,\mathcal{F}\right)  }\left(
u^{p+1}\right)  \subset\left\{  (x,[0,(p+1)\sup\Sigma_{x}(u)]),x\in
\mathcal{S}_{\mathcal{A}}^{\mathcal{F}}(u)\right\}  .
\]

\end{proof}

%%%%%%%%%%%%%%%%%%%%%%%%%%%%%%%%%%%%%%%%%%%%%%%%%%%%%%%%%%%%%%%%%%%

\section{Applications to partial differential equations\label{SPSPDE}}

\label{Sec:PDEs}
%%%%%%%%%%%%%%%%%%%%%%%%%%%%%%%%%%%%%%%%%%%%%%%%%%%%%%%%%%%%%%%%%%%

In this section we shall compute various $\left(  a,\mathcal{F}\right)
$\emph{-}singular spectra of solutions to linear and nonlinear partial
differential equations. Throughout we shall suppose that $\Lambda= ]0,1]$,
$X=\mathbb{R}^{d}$, $\mathcal{E}=\mathrm{C}^{\infty}$, $\mathcal{F}%
=\mathrm{C}^{p}$ ($1 \leq p \leq\infty$) or $\mathcal{F}={\mathcal{D}}%
^{\prime}$, $a_{\varepsilon}(r) =\varepsilon^{r}$. The results will hold for
any $(\mathcal{C},\mathcal{E},\mathcal{P})$\emph{-}algebra
\[
\mathcal{A}=\mathcal{H}_{(A,\mathcal{E},\mathcal{P})}/\mathcal{J}%
_{(I_{A},\mathcal{E},\mathcal{P})}
\]
such that $\left(  a_{\varepsilon}(r)\right)  _{\varepsilon}\in A_{+}$ for all
$r\in\mathbb{R}_{+}$ and property (\ref{SPSHypoth1}) holds.

\begin{example}
\label{Ex:deltam} The $(a, \mathrm{C}^{p})$-singular spectrum of powers of the
delta function. Given a mollifier of the form
\[
\varphi_{\varepsilon}\left(  x\right)  =\dfrac{1}{\varepsilon^{d}}%
\varphi\left(  \dfrac{x}{\varepsilon}\right)  ,\ x\in\mathbb{R}^{d}\text{
\ where }\varphi\in\mathcal{D}(\mathbb{R}^{d}), \varphi\geq0 \text{
and\ }{\textstyle\int}\varphi\left(  x\right)  dx=1,
\]
its class in ${\mathcal{A}}(\mathbb{R}^{d})$ defines the delta function
$\delta(x)$ as an element of ${\mathcal{A}}(\mathbb{R}^{d})$. Its powers are
given by ($m\in\mathbb{N}$)
\[
\delta^{m} = \big[\varphi_{\varepsilon}^{m}\big] = \Big[\dfrac{1}%
{\varepsilon^{md}}\;\varphi^{m}\left(  \dfrac{.}{\varepsilon}\right)  \Big].
\]
Clearly, the $\mathrm{C}^{0}$-singular spectrum is given by
\[
{\mathcal{S}}_{{\mathcal{A}}}^{(a,\mathrm{C}^{0})}(\delta^{m}) =
\big(0,[0,md]\big).
\]
Differentiating $\varphi^{m}(x)$ and observing that for each derivative there
is a point $x$ at which it does not vanish we see that
\[
{\mathcal{S}}_{{\mathcal{A}}}^{(a,\mathrm{C}^{k})}(\delta^{m}) =
\big(0,[0,md+k]\big).
\]

\end{example}

\begin{example}
The $(a, {\mathcal{D}}^{\prime})$-singular spectrum of powers of the delta
function. Given a test function $\psi\in{\mathcal{D}}(\mathbb{R}^{d})$, we
have
\[
\int\varphi_{\varepsilon}^{m}(x)\psi(x)\, dx = \int\dfrac{1}{\varepsilon
^{md-d}}\,\varphi^{m}(x)\psi(\varepsilon x)\, dx,
\]
thus
\[
{\mathcal{S}}_{{\mathcal{A}}}^{(a,{\mathcal{D}}^{\prime})}(\delta^{m}) =
\varnothing\ \ \mathrm{for}\ m = 1,\qquad{\mathcal{S}}_{{\mathcal{A}}%
}^{(a,{\mathcal{D}}^{\prime})}(\delta^{m}) =
\big(0,[0,md-d[\big)\ \ \mathrm{for}\ m > 1.
\]

\end{example}

%%%%%%%%%%%%%%%%%%%%%%%%%%%%%%%%%%%%%%%%%%%%%%%%%%%%%%%%%%%%%%%%%%%

\subsection{The singular spectrum of solutions to linear hyperbolic equations}

\label{Subsec:linearhyperbolic}

Consider the Cauchy problem for the $d$-dimensional linear wave equation
\begin{equation}
\label{eq:linearwave}%
\begin{array}
[c]{l}%
\partial_{t}^{2} u_{\varepsilon}(x,t) - \Delta u_{\varepsilon}(x,t) = 0,\quad
x\in\mathbb{R}^{d},\ t \in\mathbb{R}\\
u_{\varepsilon}(x,0) = u_{0\varepsilon}(x),\ \partial_{t}u_{\varepsilon}(x,0)
= u_{1\varepsilon(x)},\quad x\in\mathbb{R}^{d},
\end{array}
\end{equation}
where $u_{0\varepsilon}, u_{1\varepsilon}\in\mathrm{C}^{\infty}(\mathbb{R}%
^{d})$ represent elements $u_{0}, u_{1}$ of an algebra ${\mathcal{A}%
}(\mathbb{R}^{d})$ as outlined at the beginning of this section. Under
suitable assumptions on the ring $A$, the corresponding net of classical
smooth solutions represents a unique solution $u$ in the algebra
${\mathcal{A}}(\mathbb{R}^{d+1})$; for example, this holds in the Colombeau
case \cite{Ober1}. Let $t \to E(t) \in\mathrm{C}^{\infty}(\mathbb{R}%
:{\mathcal{E}}^{\prime}(\mathbb{R}^{d}))$ be the fundamental solution of the
Cauchy problem. Then
\[
u_{\varepsilon}(\cdot,t) = \dfrac{d}{dt}E(t) \ast\varepsilon^{r}
u_{0\varepsilon} + E(t)\ast\varepsilon^{r} u_{1\varepsilon}.
\]
If for some $r\geq0$ and $u_{0}\in{\mathcal{D}}^{\prime}(\mathbb{R}^{d})$,
\[
\int\varepsilon^{r} u_{0\varepsilon}(x)\psi(x)\,dx \to\langle u_{0}%
,\psi\rangle
\]
for all $\psi\in{\mathcal{D}}(\mathbb{R}^{d})$, then
\[
\int\varepsilon^{r} u_{\varepsilon}(x)\psi(x)\,dx = \langle E(t)\ast
\varepsilon^{r} u_{0\varepsilon},\psi\rangle= \langle\varepsilon^{r}
u_{0\varepsilon},\check{E}(t)\ast\psi\rangle\to\langle u_{0},\check{E}%
(t)\ast\psi\rangle
\]
for all $\psi\in{\mathcal{D}}(\mathbb{R}^{d})$ and $t\in\mathbb{R}$ as well.
We arrive at the following assertion.

\begin{proposition}
Assume that ${\mathcal{S}}_{{\mathcal{A}}}^{(a,{\mathcal{D}}^{\prime})}%
(u_{0})$ and ${\mathcal{S}}_{{\mathcal{A}}}^{(a,{\mathcal{D}}^{\prime})}%
(u_{1})$ are contained in $\mathbb{R}^{d}\times I$, where $I = \varnothing$,
$I = [0,r[$ or $I = [0,r]$ for some $r$, $0\leq r \leq\infty$. Let
$u\in{\mathcal{A}}(\mathbb{R}^{d+1})$ be the solution to the linear wave
equation (\ref{eq:linearwave}). Then ${\mathcal{S}}_{{\mathcal{A}}%
}^{(a,{\mathcal{D}}^{\prime})}(u(\cdot,t)) \subset\mathbb{R}^{d}\times I$ for
all $t\in\mathbb{R}$.
\end{proposition}

This upper bound may or may not be reached, depending on the effects of finite
propagation speed or the Huyghens principle in odd space dimension $d \geq3$.
We just illustrate some of the possible effects for the one-dimensional wave
equation with powers of delta functions as initial data. Thus we consider the
problem
\begin{equation}
\label{eq:linearwave1d}%
\begin{array}
[c]{l}%
\partial_{t}^{2} u_{\varepsilon}(x,t) - \partial_{x}^{2} u_{\varepsilon}(x,t)
= 0,\quad x\in\mathbb{R},\ t \in\mathbb{R}\\
u_{\varepsilon}(x,0) = c_{0}\varphi_{\varepsilon}^{m}(x),\ \partial
_{t}u_{\varepsilon}(x,0) = c_{1}\varphi_{\varepsilon}^{n}(x),\quad
x\in\mathbb{R},
\end{array}
\end{equation}
where $\varphi$ is a mollifier as in Example~\ref{Ex:deltam} and $c_{0}, c_{1}
\in\mathbb{R}$. The solution to (\ref{eq:linearwave1d}) is given by
\[
u_{\varepsilon}(x,t) = \dfrac{c_{0}}{2}\big(\varphi_{\varepsilon}^{m}(x-t) +
\varphi_{\varepsilon}^{m}(x+t)\big) + \dfrac{c_{1}}{2}\int_{x-t}^{x+t}%
\varphi_{\varepsilon}^{n}(y)\,dy.
\]
We observe that $u_{\varepsilon}(x,t) = 0$ for sufficiently small
$\varepsilon$ when $|x| > |t|$, that is, outside the light cone, and
$u_{\varepsilon}(x,t) = \mathrm{sign}\vspace{0.5pt}(t)\frac{c_{1}}%
{2}\varepsilon^{n-1}\|\varphi^{n}\|_{\mathrm{L}^{1}(\mathbb{R})}$ for
sufficiently small $\varepsilon$ when $|x| < |t|$.

\begin{example}
If in equation~(\ref{eq:linearwave1d}) $c_{0} \neq0$, $c_{1} = 0$ then
\[
{\mathcal{S}}_{{\mathcal{A}}}^{(a,{\mathcal{D}}^{\prime})}(u) = \{(x,t,r): |x|
= |t|, 0\leq r < m-1\}
\]
with the provision that ${\mathcal{S}}_{{\mathcal{A}}}^{(a,{\mathcal{D}%
}^{\prime})}(u) = \varnothing$ when $m = 1$. If in
equation~(\ref{eq:linearwave1d}) $c_{0} = 0$, $c_{1} \neq0$ then
\[
{\mathcal{S}}_{{\mathcal{A}}}^{(a,{\mathcal{D}}^{\prime})}(u) = \{(x,t,r): |x|
\leq|t|, 0\leq r < n-1\}.
\]
When both $c_{0}$ and $c_{1}$ are nonzero the singular spectrum is obtained as
the union of the two spectra. For the $\mathrm{C}^{0}$-singular spectrum the
following results hold: If in equation~(\ref{eq:linearwave1d}) $c_{0} \neq0$,
$c_{1} = 0$ then
\[
{\mathcal{S}}_{{\mathcal{A}}}^{(a,\mathrm{C}^{0})}(u) = \{(x,t,r): |x| = |t|,
0\leq r \leq m\}.
\]
If $c_{0} = 0$, $c_{1} \neq0$ then
\[
{\mathcal{S}}_{{\mathcal{A}}}^{(a,\mathrm{C}^{0}))}(u) = \{(x,t,r): |x| < |t|,
0 \leq r < n-1\}\cup\{(x,t,r): |x| = |t|, 0\leq r \leq n-1\}.
\]

\end{example}

%%%%%%%%%%%%%%%%%%%%%%%%%%%%%%%%%%%%%%%%%%%%%%%%%%%%%%%%%%%%%%%%%%%

\subsection{The singular spectrum of solutions to semilinear hyperbolic
equations}

\label{Subsec:semilinearhyperbolic}

In this subsection we study the paradigmatic case of a semilinear transport
equation
\begin{equation}
\label{eq:transport}%
\begin{array}
[c]{l}%
\partial_{t} u_{\varepsilon}(x,t) + \lambda(x,t)\partial_{x} u_{\varepsilon
}(x,t) = F(u_{\varepsilon}(x,t)),\quad x\in\mathbb{R},\ t \in\mathbb{R}\\
u_{\varepsilon}(x,0) = u_{0\varepsilon}(x),\quad x\in\mathbb{R}%
\end{array}
\end{equation}
where $\lambda$ and $F$ are smooth functions of their arguments. In this
situation, the singular spectrum of the initial data may be decreased or
increased, depending on the function $F$. We observe that by a change of
coordinates we may assume without loss of generality that $\lambda\equiv0$. In
fact, denote by $s\to\gamma(x,t,s)$ the characteristic curve of
(\ref{eq:transport}) passing through the point $x$ at time $s = t$, that is
the solution to
\[
\dfrac{d}{ds}\gamma(x,t,s) = \lambda(\gamma(x,t,s),s),\qquad\gamma(x,t,t) =
x.
\]
The function $v(y,s) = u(\gamma(y,0,s),s)$ is a solution of the initial value
problem
\[
\partial_{s} v(y,s) = F(v(y,s),\qquad v(y,0) = u_{0}(y),
\]
at least as long as the characteristic curves exist.

\begin{example}
(The dissipative case) The equation
\[%
\begin{array}
[c]{l}%
\partial_{t} u_{\varepsilon}(x,t) = -u^{3}_{\varepsilon}(x,t), \quad
x\in\mathbb{R},\ t > 0\\
u_{\varepsilon}(x,0) = u_{0\varepsilon}(x), \quad x\in\mathbb{R}%
\end{array}
\]
has the solution
\[
u_{\varepsilon}(x,t) = \frac{u_{0\varepsilon}(x)}{\sqrt{2tu^{2}_{0\varepsilon
}(x) + 1}} = \frac{1}{\sqrt{2t + 1/u^{2}_{0\varepsilon}(x)}}.
\]
When the initial data are given by a power of the delta function,
$u_{0\varepsilon}(x) = \varphi_{\varepsilon}^{m}(x)$, the solution formula
shows that $u_{\varepsilon}(x,t)$ is a bounded function (uniformly in
$\varepsilon$) and supported on the line $\{x = 0\}$. Thus $u_{\varepsilon
}(x,t)$ converges to zero in ${\mathcal{D}}^{\prime}(\mathbb{R}\times
]0,\infty[)$, and so
\[
{\mathcal{S}}_{{\mathcal{A}}}^{(a,{\mathcal{D}}^{\prime})}(u_{0}) =
\big(0,[0,m-1[\big),\qquad{\mathcal{S}}_{{\mathcal{A}}}^{(a,{\mathcal{D}%
}^{\prime})}(u) = \varnothing.
\]

\end{example}

\begin{example}
The equation
\[%
\begin{array}
[c]{l}%
\partial_{t}u_{\varepsilon}(x,t)=\sqrt{1+u_{\varepsilon}^{2}(x,t)},\quad
x\in\mathbb{R},\ t>0\\
u_{\varepsilon}(x,0)=u_{0\varepsilon}(x),\quad x\in\mathbb{R}%
\end{array}
\]
has the solution
\[
u_{\varepsilon}(x,t)=u_{0\varepsilon}(x)\cosh t+\sqrt{1+u_{0\varepsilon}%
^{2}(x)}\,\sinh t.
\]
We first take a delta function as initial value, that is, $u_{0\varepsilon
}(x)=\varphi_{\varepsilon}(x)$. Then
\begin{align*}
\iint u_{\varepsilon}(x,t)\psi(x,t)\,dxdt  &  =\iint\Big(\varphi(x)\cosh
t+\sqrt{\varepsilon^{2}+\varphi^{2}(x)}\,\sinh t\Big)\psi(\varepsilon
x,t)\,dxdt\\
&  \rightarrow\iint\Big(\varphi(x)\cosh t+|\varphi(x)|\sinh t\Big)\psi
(0,t)\,dxdt
\end{align*}
for $\psi\in{\mathcal{D}}(\mathbb{R}^{2})$. Thus in this case
\[
{\mathcal{S}}_{{\mathcal{A}}}^{(a,{\mathcal{D}}^{\prime})}(u_{0}%
)={\mathcal{S}}_{{\mathcal{A}}}^{(a,{\mathcal{D}}^{\prime})}(u)=\varnothing.
\]
On the other hand, taking the derivative of a delta function as initial value,
$u_{0\varepsilon}(x)=\varphi_{\varepsilon}^{\prime}(x)$, a similar calculation
shows that
\[
\iint u_{\varepsilon}(x,t)\psi(x,t)\,dxdt=\iint\Big(\varphi(x)\cosh
t+\dfrac{1}{\varepsilon}\sqrt{\varepsilon^{4}+(\varphi^{\prime})^{2}%
(x)}\,\sinh t\Big)\psi(\varepsilon x,t)\,dxdt
\]
and so
\[
{\mathcal{S}}_{{\mathcal{A}}}^{(a,{\mathcal{D}}^{\prime})}(u_{0}%
)=\varnothing,\qquad{\mathcal{S}}_{{\mathcal{A}}}^{(a,{\mathcal{D}}^{\prime}%
)}(u)=\{(0,t,r):t>0,0\leq r<1\}.
\]

\end{example}

The next example shows that it is quite possible for the singular spectrum to
increase with time.

\begin{example}
The equation
\[%
\begin{array}
[c]{l}%
\partial_{t} u_{\varepsilon}(x,t) = \big(u_{\varepsilon}(x,t)+1)\big)\log
\big(u_{\varepsilon}(x,t)+1\big),\quad x\in\mathbb{R},\ t > 0\\
u_{\varepsilon}(x,0) = u_{0\varepsilon}(x),\quad x\in\mathbb{R}%
\end{array}
\]
has the solution
\[
u_{\varepsilon}(x,t) = \big(u_{0\varepsilon}(x)+1\big)^{e^{t}},
\]
provided $u_{0\varepsilon} > -1$ in which case the function on the right hand
side of the differential equation is smooth in the relevant region. To
demonstrate the effect, we take a power of the delta function as initial
value, that is $u_{0\varepsilon}(x) = \varphi^{m}_{\varepsilon}(x)$. Then
\[
{\mathcal{S}}_{{\mathcal{A}}}^{(a,{\mathcal{D}}^{\prime})}(u_{0}) = \{(0, r):
0 \leq r < m-1\}, \qquad{\mathcal{S}}_{{\mathcal{A}}}^{(a,{\mathcal{D}%
}^{\prime})}(u) = \{(0,t,r): t>0, 0 \leq r < me^{t}-1\}.
\]

\end{example}

In situations where blow-up in finite time occurs, microlocal asymptotic
methods allow to extract information beyond the point of blow-up. This can be
done by regularizing the initial data and truncating the nonlinear term. We
demonstrate this in a simple situation.

\begin{example}
Formally, we wish to treat the initial value problem
\[%
\begin{array}
[c]{l}%
\partial_{t}u(x,t)=u^{2}(x,t),\quad x\in\mathbb{R},\ t>0\\
u(x,0)=H(x),\quad x\in\mathbb{R}%
\end{array}
\]
where $H$ denotes the Heaviside function. Clearly, the local solution
$u(x,t)=H(x)/(1-t)$ blows up at time $t=1$ when $x>0$. Choose $\chi
_{\varepsilon}\in\mathrm{C}^{\infty}\left(  \mathbb{R}\right)  $ with%
\[
0\leq\chi_{\varepsilon}(z)\leq1\ ;\ \chi_{\varepsilon}(z)=1\text{ if }%
|z|\leq\varepsilon^{-s}\,,\ \chi_{\varepsilon}(z)=0\ \text{if }|z|\geq
1+\varepsilon^{-s}\,,\ s>0.
\]
Further, let $H_{\varepsilon}(x)=H\ast\varphi_{\varepsilon}(x)$ where
$\varphi_{\varepsilon}$ is a mollifier as in Example~\ref{Ex:deltam}. We
consider the regularized problem
\[%
\begin{array}
[c]{l}%
\partial_{t}u_{\varepsilon}(x,t)=\chi_{\varepsilon}\big(u_{\varepsilon
}(x,t)\big)u_{\varepsilon}^{2}(x,t),\quad x\in\mathbb{R},\ t>0\\
u_{\varepsilon}(x,0)=H_{\varepsilon}(x),\quad x\in\mathbb{R}.
\end{array}
\]
When $x<0$ and $\varepsilon$ is sufficiently small, $u_{\varepsilon}(x,t)=0$
for all $t\geq0$. For $x>0$, $u_{\varepsilon}(x,t)=1/(1-t)$ as long as
$t\leq1-\varepsilon^{s}$. The cut-off function is chosen in such a way that
$|\chi_{\varepsilon}(z)z^{2}|\leq(1+\varepsilon^{-s})^{2}$ for all
$z\in\mathbb{R}$. Therefore,
\[
\partial_{t}u_{\varepsilon}\leq(1+\varepsilon^{-s})^{2}%
\ \mbox{always\ and}\ \partial_{t}u_{\varepsilon}%
=0\ \mbox{when}\ |u_{\varepsilon}|\geq1+\varepsilon^{-s}.
\]
Continuing the regularized solution beyond time $t=1-\varepsilon^{s}$, we
infer by combining the two inequalities that $\varepsilon^{-s}\leq
u_{\varepsilon}(x,t)\leq1+\varepsilon^{-s}$ for $t\geq1-\varepsilon^{s}$ when
$x>0$ and $\varepsilon$ is sufficiently small. Finally, as long as $t<1$, the
regularized solution remains bounded with respect to $\varepsilon$ near
$\left(  0,t\right)  $ for $\varepsilon$ small enough; after $t=1$, the
asymptotic growth of order $\varepsilon^{-s}$ spills over into any
neighborhood of every point $(x,t)$ for $x\geq0$.

Collecting all previous estimates, we obtain the following $\mathrm{C}^{0}%
$-singular support and $\left(  a,\mathrm{C}^{0}\right)  $-singular spectrum
(for $a_{\varepsilon}(r)=\varepsilon^{r}$) of $u=\left[  u_{\varepsilon
}\right]  $:%
\[
{\mathcal{S}}_{{\mathcal{A}}}^{\mathrm{C}^{0}}(u)={\mathcal{S}}_{1}%
(u)\cup{\mathcal{S}}_{2}(u)\text{ with }{\mathcal{S}}_{1}(u)=\{(0,t):0\leq
t<1\}~;~{\mathcal{S}}_{2}(u)=\{(x,t):x\geq0,t\geq1\},
\]%
\[
{\mathcal{S}}_{{\mathcal{A}}}^{(a,\mathrm{C}^{0})}(u)=\left(  {\mathcal{S}%
}_{1}(u)\times\left\{  0\right\}  \right)  \cup\left(  {\mathcal{S}}%
_{2}(u)\times\left[  0,s\right]  \right)  .
\]
The $\mathrm{C}^{0}$-singularities (resp. $\left(  a,\mathrm{C}^{0}\right)
$-singularities) of $u$ are described by means of two sets: ${\mathcal{S}}%
_{1}(u)$ and ${\mathcal{S}}_{2}(u)$ (resp. ${\mathcal{S}}_{1}(u)\times\left\{
0\right\}  $ and ${\mathcal{S}}_{2}(u)\times\left[  0,s\right]  $). The set
${\mathcal{S}}_{1}(u)$ (resp. ${\mathcal{S}}_{1}(u)\times\left\{  0\right\}
$) is related to the data $\mathrm{C}^{0}$ (resp. $\left(  a,\mathrm{C}%
^{0}\right)  $)-singularity. The set ${\mathcal{S}}_{2}(u)$ (resp.
${\mathcal{S}}_{2}(u)\times\left[  0,s\right]  $) is related to the
singularity due to the nonlinearity of the equation giving the blow-up at
$t=1$. The blow-up locus is the edge $\left\{  x\geq0,t=1\right\}  $ of
${\mathcal{S}}_{2}(u)$ and the strength of the blow-up is measured by the
length $s$ of the fiber $\left[  0,s\right]  $ above each point of the blow-up
locus. This length is closely related to the diameter of the support of the
regularizing function $\chi_{\varepsilon}$ and depends essentially on the
nature of the blow-up: Changing simultaneously the scales of the
regularization and of the cut-off (i.e. replacing $\varepsilon$ by some
function $h(\varepsilon) \to0$ in the definition of $\varphi_{\varepsilon}$
and $\chi_{\varepsilon}$) does not change the fiber and characterizes a sort
of moderateness of the strength of the blow-up.
\end{example}

%%%%%%%%%%%%%%%%%%%%%%%%%%%%%%%%%%%%%%%%%%%%%%%%%%%%%%%%%%%%%%%%%%%

\subsection{The strength of a singularity and the sum law}

\label{Subsec:strentgthsum}

When studying the propagation and interaction of singularities in semilinear
hyperbolic systems, Rauch and Reed \cite{RauchReed} defined the strength of a
singularity of a piecewise smooth function. We recall this notion in the
one-dimensional case. Assume that the function $f: \mathbb{R }\to\mathbb{R}$
is smooth on $]-\infty, x_{0}]$ and on $[x_{0}, \infty[$ for some point
$x_{0}\in\mathbb{R}$. The \emph{strength of the singularity of $f$ at $x_{0}$}
is the order of the highest derivative which is still continuous across
$x_{0}$. For example, if $f$ is continuous with a jump in the first derivative
at $x_{0}$, the order is $0$; if $f$ has a jump at $x_{0}$, the order is $-1$.
Travers \cite{Travers} later generalized this notion to include delta
functions. Slightly deviating from her definition, but in line with the one of
\cite{RauchReed}, we define the strength of singularity of the $k$-th
derivative of a delta function at $x_{0}$, $\partial_{x}^{k}\delta(x-x_{0})$,
by $-k-2$.

The significance of these definitions is seen in the description of what Rauch
and Reed termed \emph{anomalous singularities} in semilinear hyperbolic
systems. We demonstrate the effect in a paradigmatic example, also due to
\cite{RauchReed}, the $(3\times3)$-system
\begin{equation}
\label{eq:RSsystem}%
\begin{array}
[c]{rclcl}%
(\partial_{t} + \partial_{x}) u(x,t) & = & 0, & \quad & u(x,0) = u_{0}(x)\\
(\partial_{t} - \partial_{x}) v(x,t) & = & 0, & \quad & v(x,0) = v_{0}(x)\\
\partial_{t} w(x,t) & = & u(x,t)v(x,t), & \quad & w(x,0) = 0
\end{array}
\end{equation}
Assume that $u_{0}$ has a singularity of strength $n_{1} \geq-1$ at $x_{1} =
-1$ and $v_{0}$ has a singularity of strength $n_{2} \geq-1$ at $x_{2} = +1$.
The characteristic curves emanating from $x_{1}$ and $x_{2}$ are straight
lines intersecting at the point $x=0$, $t = 1$. Rauch and Reed showed that, in
general, the third component $w$ will have a singularity of strength $n_{3} =
n_{1} + n_{2} + 2$ along the half-ray $\{(0,t):t \geq1\}$. This half-ray does
not connect backwards to a singularity in the initial data for $w$, hence the
term \emph{anomalous singularity}. The formula $n_{3} = n_{1} + n_{2} + 2$ is
called the \emph{sum law}. Travers extended this result to the case where
$u_{0}$ and $v_{0}$ were given as derivatives of delta functions at $x_{1}$
and $x_{2}$. We are going to further generalize this result to powers of delta
functions, after establishing the relation between the strength of a
singularity of a function $f$ at $x_{0}$ and the singular spectrum of
$f\ast\varphi_{\varepsilon}$.

We consider a function $f: \mathbb{R }\to\mathbb{R}$ which is smooth on
$]-\infty, x_{0}]$ and on $[x_{0}, \infty[$ for some point $x_{0}\in
\mathbb{R}$; actually only the local behavior near $x_{0}$ is relevant. We fix
a mollifier $\varphi_{\varepsilon}(x) = \frac{1}{\varepsilon}\varphi(\frac
{x}{\varepsilon})$ as in Example~\ref{Ex:deltam} and denote the corresponding
embedding of ${\mathcal{D}}^{\prime}(\mathbb{R})$ into the $(\mathcal{C}%
,\mathcal{E},\mathcal{P})$\emph{-}algebra ${\mathcal{A}}(\mathbb{R})$ by
$\iota$. In particular, $\iota(f) = [f\ast\varphi_{\varepsilon}]$.

If $f$ is continuous at $x_{0}$, then $\lim_{\varepsilon\to0}f\ast
\varphi_{\varepsilon}= f$ in $\mathrm{C}^{0}$. If $f$ has a jump $x_{0}$, this
limit does not exist in $\mathrm{C}^{0}$, but $\lim_{\varepsilon\to
0}\varepsilon^{r}f\ast\varphi_{\varepsilon}= 0$ in $\mathrm{C}^{0}$ for every
$r>0$. We have the following result.

\begin{proposition}
\label{Prop:strength} Let $x_{0} \in\mathbb{R}$. If $f: \mathbb{R }%
\to\mathbb{R}$ is a smooth function on $]-\infty, x_{0}]$ and on $[x_{0},
\infty[$ or $f(x) = \partial_{x}^{k}\delta(x-x_{0})$ for some $k \in
\mathbb{N}$, then the strength of the singularity of $f$ at $x_{0}$ is $-n$ if
and only if
\[
\Sigma_{(a,\mathrm{C}^{1}), x_{0}}\big(\iota(f)\big) = [0, n].
\]
Here $n\in\mathbb{N}$ and $a_{\varepsilon}(r) = \varepsilon^{r}$.
\end{proposition}

\begin{proof}
When $n=0$, the function $f$ is continuous and its derivative has a jump at
$x_{0}$. From what was said before Proposition~\ref{Prop:strength} it follows
that $\Sigma_{(a,\mathrm{C}^{1}), x_{0}}\big(\iota(f)\big) = \{0\}$. When
$n=1$, the function $f$ has a jump itself at $x_{0}$ and its distributional
derivative contains a delta function part. Thus $\lim_{\varepsilon\to
0}\varepsilon^{r}f\ast\varphi_{\varepsilon}= 0$ in $\mathrm{C}^{0}$ for every
$r>0$ and $\lim_{\varepsilon\to0}\varepsilon^{r}\partial_{x}f\ast
\varphi_{\varepsilon}= 0$ in $\mathrm{C}^{0}$ for every $r>1$, and neither of
the two limits exists for smaller $r$. Therefore, $\Sigma_{(a,\mathrm{C}^{1}),
x_{0}}\big(\iota(f)\big) = [0,1]$. When $n\geq2$, $f(x) = \partial_{x}%
^{n-2}\delta(x-x_{0})$ and the assertion is straightforward.
\end{proof}

We shall now return to the model equation (\ref{eq:RSsystem}) and demonstrate
that the sum law remains valid when the initial data are powers of delta
functions. We work in suitable $(\mathcal{C},\mathcal{E},\mathcal{P})$%
\emph{-}algebras ${\mathcal{A}}(\mathbb{R})$ and ${\mathcal{A}}(\mathbb{R}%
^{2})$ in which the initial value problem (\ref{eq:RSsystem}) can be uniquely
solved (see the discussion at the beginning of
Subsection~\ref{Subsec:linearhyperbolic}). We still consider the scale
$a_{\varepsilon}(r) = \varepsilon^{r}$.

\begin{proposition}
Let $u_{0}(x) = \delta^{m}(x+1)$, $v_{0}(x) = \delta^{n}(x-1)$ for some
$m,n\in\mathbb{N}^{\ast}$. Let $w \in{\mathcal{A}}(\mathbb{R}^{2})$ be the
third component of the solution to problem (\ref{eq:RSsystem}). Then $w(x,t)$
vanishes at all points $(x,t)$ with $x\neq0$ as well as $(0,t)$ with $t < 1$,
and
\[
\Sigma_{(a,\mathrm{C}^{1}), (0,t)}\big(w\big) \subset[0, m+n]
\]
for $t\geq1$.
\end{proposition}

\begin{proof}
A representative of $w$ is given by
\[
w_{\varepsilon}(x,t) = \int_{0}^{t}\varphi_{\varepsilon}^{m}(x+1-s)\varphi
_{\varepsilon}^{n}(x-1+s)\,ds.
\]
The fact that the mollifier $\varphi$ has compact support entails that
$w_{\varepsilon}(x,t)$ vanishes for sufficiently small $\varepsilon$ whenever
$x\neq0$ or $t < 1$. We have
\begin{align*}
w_{\varepsilon}(x,t)  &  = \int_{0}^{t} \dfrac{1}{\varepsilon^{m+n}}%
\;\varphi^{m}\Big(\dfrac{x+1-s}{\varepsilon}\Big)\varphi^{n}\Big(\dfrac
{x-1+s}{\varepsilon}\Big)\,ds,\\
\partial_{t}w_{\varepsilon}(x,t)  &  = \dfrac{1}{\varepsilon^{m+n}}%
\;\varphi^{m}\Big(\dfrac{x+1-t}{\varepsilon}\Big)\varphi^{n}\Big(\dfrac
{x-1+t}{\varepsilon}\Big),\\
\partial_{x}w_{\varepsilon}(x,t)  &  = \int_{0}^{t} \dfrac{m}{\varepsilon
^{m+n+1}}\; \varphi^{m-1}\Big(\dfrac{x+1-s}{\varepsilon}\Big)\varphi^{\prime
}\Big(\dfrac{x+1-s}{\varepsilon}\Big)\varphi^{n}\Big(\dfrac{x-1+s}%
{\varepsilon}\Big)\,ds\\
&  +\ \int_{0}^{t} \dfrac{n}{\varepsilon^{m+n+1}}\; \varphi^{m}\Big(\dfrac
{x+1-s}{\varepsilon}\Big)\varphi^{n-1}\Big(\dfrac{x-1+s}{\varepsilon
}\Big)\varphi^{\prime}\Big(\dfrac{x-1-s}{\varepsilon}\Big)\,ds.
\end{align*}
If the support of $\varphi$ is contained in an interval $[-\kappa,\kappa]$,
say, then the $t$-integrations extend at most from $x+1-\kappa\varepsilon$ to
$x+1+\kappa\varepsilon$ at fixed $x$. Therefore, all terms converge to zero
uniformly on $\mathbb{R}^{2}$ when multiplied by $\varepsilon^{r}$ with
$r>m+n$. This proves the assertion.
\end{proof}

Using the correspondence between the singular spectrum and the strength of a
singularity formulated in Proposition~\ref{Prop:strength}, as well as
Example~\ref{Ex:deltam}, we may say that the strength of the singularity of
$\delta^{m}(x+1)$ at $x_{0}=-1$ is $n_{1} = -m-1$, while the strength of the
singularity of $\delta^{n}(x-1)$ at $x_{0}=+1$ is $n_{2} = -n-1$. The strength
of the singularity of the solution $w$ at points $(0,t)$ with $t\geq1$ is
$-m-n = n_{1} + n_{2} + 2$ and is seen to satisfy the sum law.
%
%%%%%%%%%%%%%%%%%%%%%%%%%%%%%%%%%%%%%%%%%%%%%%%%%%%%%%%%%%%%%%%%%%%

\subsection{Regular Colombeau generalized functions}

\label{Subsec:regColombeau} The subsheaf ${\mathcal{G}}^{\infty}$ of
\emph{regular Colombeau functions} of the sheaf ${\mathcal{G}}$ is defined as
follows \cite{Ober1}: Given an open subset $\Omega$ of $\mathbb{R}^{d}$, the
algebra ${\mathcal{G}}^{\infty}(\Omega)$ comprises those elements $u$ of
${\mathcal{G}}(\Omega)$ whose representatives $(u_{\varepsilon})_{\varepsilon
}$ satisfy the condition
\begin{equation}
\label{eq:Ginfty}\forall K \Subset\Omega\ \exists m \in\mathbb{N}\ \forall
l\in\mathbb{N}:p_{K,l}(u_{\varepsilon}) = o(\varepsilon^{-m})\ \mathrm{as}%
\ \varepsilon\to0.
\end{equation}
The decisive property is that the bound of order $\varepsilon^{-m}$ is uniform
with respect to the order of derivation on compact sets. The algebra
${\mathcal{G}}^{\infty}(\Omega)$ satisfies ${\mathcal{G}}^{\infty}(\Omega)
\cap{\mathcal{D}}^{\prime}(\Omega) = \mathrm{C}^{\infty}(\Omega)$ and forms
the basis for the investigation of hypoellipticity of linear partial
differential operators in the Colombeau framework. We are going to
characterize the ${\mathcal{G}}^{\infty}$-property in terms of the
$\mathrm{C}^{\infty}$-singular spectrum. The scale $a$ is still given by
$a_{\varepsilon}(r) = \varepsilon^{r}$.

\begin{proposition}
Let $u \in{\mathcal{G}}(\Omega)$. Then $u$ belongs to ${\mathcal{G}}^{\infty
}(\Omega)$ if and only if
\[
\Sigma_{(a,\mathrm{C}^{\infty}),x}\big(u\big) \neq\mathbb{R}_{+}
\]
for all $x\in\Omega$.
\end{proposition}

\begin{proof}
If $u\in{\mathcal{G}}^{\infty}(\Omega)$, $x\in\Omega$ and $V_{x}$ is a
relatively compact open neighborhood of $x$, property (\ref{eq:Ginfty}) says
that there is $m\in\mathbb{N}$ such that $\lim_{\varepsilon\to0}%
\varepsilon^{m} u_{\varepsilon}= 0$ in $\mathrm{C}^{\infty}(V_{x})$. Thus
$\Sigma_{(a,\mathrm{C}^{\infty}),x}\big(u\big) \neq\mathbb{R}_{+}$.
Conversely, if $\Sigma_{(a,\mathrm{C}^{\infty}),x}\big(u\big) \neq
\mathbb{R}_{+}$ we can find an open neighborhood $V_{x}$ of $x$ and
$m(x)\in\mathbb{N}$ such that $\lim_{\varepsilon\to0}\varepsilon^{r}
u_{\varepsilon}= 0$ in $\mathrm{C}^{\infty}(V_{x})$ for all $r\geq m$. Any
compact set $K$ can be covered by finitely many such neighborhoods. Letting
$m$ be the maximum of the numbers $m(x)$ involved, we obtain property
(\ref{eq:Ginfty}).
\end{proof}

In relation with regularity theory of solutions to nonlinear partial
differential equations, a further subalgebra of ${\mathcal{G}}(\Omega)$ has
been introduced in \cite{OberBiaVolume} -- the algebra of Colombeau functions
of \emph{total slow scale type}. It consists of those elements $u$ of
${\mathcal{G}}(\Omega)$ whose representatives $(u_{\varepsilon})_{\varepsilon
}$ satisfy the condition
\begin{equation}
\label{eq:slowscale}\forall K \Subset\Omega\ \forall r > 0\ \forall
l\in\mathbb{N}:p_{K,l}(u_{\varepsilon}) = o(\varepsilon^{-r})\ \mathrm{as}%
\ \varepsilon\to0.
\end{equation}
The term \emph{slow scale} refers to the fact that the growth is slower than
any negative power of $\varepsilon$ as $\varepsilon\to0$. This property can
again be characterized by means of the singular spectrum.

\begin{proposition}
An element $u \in{\mathcal{G}}(\Omega)$ is of total slow scale type if and
only if
\[
\Sigma_{(a,\mathrm{C}^{\infty}),x}\big(u\big) \subset\{0\}
\]
for all $x\in\Omega$.
\end{proposition}

\begin{proof}
If $u$ is of total slow scale type, $x\in\Omega$ and $V_{x}$ is a relatively
compact open neighborhood of $x$, property (\ref{eq:slowscale}) implies that
$\lim_{\varepsilon\rightarrow0}\varepsilon^{s}u_{\varepsilon}=0$ in
$\mathrm{C}^{\infty}(V_{x})$ for every $s>0$. Thus $\Sigma_{(a,\mathrm{C}%
^{\infty}),x}\big(u\big)\subset\{0\}$. To prove the converse, we take a
compact subset $K$ and $r>0$ and cover $K$ by finitely many neighborhoods
$V_{x}$ of points $x\in K$ such that $\lim_{\varepsilon\rightarrow
0}\varepsilon^{r}u_{\varepsilon}=0$ in $\mathrm{C}^{\infty}(V_{x})$. Then
property (\ref{eq:slowscale}) follows.
\end{proof}


\begin{thebibliography}{99}                                                                                               %


\bibitem {AraBia}Aragona,~J., Biagioni,~H.: Intrinsic definition of the
Colombeau algebra of generalized functions.\ Anal. Math. \textbf{17}, 75--132 (1991)

\bibitem {JFC3}Colombeau,~J.F.: Multiplication of Distributions: a tool in
mathematics, numerical engineering and theoretical physics. Lecture Notes in
Mathematics, vol. 1532. Springer-Verlag, Berlin (1992)

\bibitem {ADEmbed}Delcroix,~A: Remarks on the embedding of spaces of
distributions into spaces of Colombeau generalized functions. Novi Sad J.
Math. \textbf{35}(2), 27--40 (2005)

\bibitem {ADRReg}Delcroix,~A.: Regular rapidly decreasing nonlinear
generalized functions. Application to microlocal regularity. J. Math. Anal.
Appl. \textbf{327}, 564--584 (2007)

\bibitem {GaGrOb}Garetto,~C., Gramchev,~T., Oberguggenberger,~M.:
Pseudo-Differential operators and regularity theory. Electron J. Diff. Eqns.
\textbf{2005}(116), 1--43 (2005)

\bibitem {GarHorm}Garetto,~C., H\"{o}rmann,~G.: Microlocal analysis of
generalized functions: pseudodifferential techniques and propagation of
singularities. Proc. Edinb. Math. Soc. \textbf{48}, 603--629 (2005)

\bibitem {GKOS}Grosser,~M., Kunzinger,~M., Oberguggenberger,~M.,
Steinbauer,~R.: Geometric Theory of Generalized Functions with Applications to
General Relativity. Kluwer Academic Publ., Dordrecht (2001)

\bibitem {HorPDOT1}H\"{o}rmander,~L.: The Analysis of Linear Partial
Differential Operators, I:~ Distribution Theory and Fourier Analysis.
Grundlehren der mathematischen Wissenschaften, vol. 256, 2nd edition.
Springer-Verlag, Berlin (1990)

\bibitem {HorDeH}H\"{o}rmann,~G.,~De~Hoop,~M.V.: Microlocal analysis and
global solutions for some hyperbolic equations with discontinuous
coefficients. Acta Appl. Math. \textbf{67}, 173--224 (2001)

\bibitem {HorKun}H\"{o}rmann, G.,~Kunzinger, M.: Microlocal properties of
basic operations in Colombeau algebras. J.\ Math.\ Anal.\ Appl. \textbf{261},
254--270 (2001)

\bibitem {HorObPil}H\"{o}rmann, G., Oberguggenberger, M., Pilipovi\'{c}, S.:
Microlocal hypoellipticity of linear differential operators with generalized
functions as coefficients. Trans. Amer. Math. Soc. \textbf{358}, 3363--3383 (2006)

\bibitem {JAM0}Marti,~J.A.: $\left(  \mathcal{C},\mathcal{E},\mathcal{P}%
\right)  $-Sheaf structure and applications. In: Grosser, M., H\"{o}rmann, G.,
Kunzinger, M., Oberguggenberger, M. (Eds.) Nonlinear Theory of Generalized
Functions. Chapman \& Hall/CRC Research Notes in Mathematics, vol. 401,
pp.175--186. Boca Raton (1999)

\bibitem {JAM1}Marti,~J.A.: Nonlinear algebraic analysis of delta shock wave
to Burgers' equation. Pacific J. Math. \textbf{210}(1), 165--187 (2003)

\bibitem {JAM3}Marti,~J.A.: $\mathcal{G}^{L}$-microanalysis of generalized
functions. Integral Transf. Spec. Funct. \textbf{2}--\textbf{3}, 119--125 (2006)

\bibitem {NePiSc}Nedeljkov,~M., Pilipovi\'{c},~S., Scarpal\'{e}zos,~D.: The
linear theory of Colombeau generalized functions. Pitman Research Notes in
Mathematics, vol. 385. Longman Scientific {\&} Technical, Harlow (1998)

\bibitem {Ober1}Oberguggenberger,~M.: Multiplication of Distributions and
Applications to Partial Differential Equations. Pitman Research Notes in
Mathematics, vol. 259. Longman Scientific {\&} Technical, Harlow (1992)

\bibitem {OberBiaVolume}Oberguggenberger,~M.: Generalized solutions to
nonlinear wave equations. Matem\'{a}tica Contempor\^{a}nea \textbf{27},
169--187 (2004)

\bibitem {RauchReed}Rauch, J., Reed, M.: Jump discontinuities of semilinear,
strictly hyperbolic equations in two variables: Creation and propagation.
Comm. Math. Phys. \textbf{81}, 203-227 (1981)

\bibitem {Scarpa1}Scarpal\'{e}zos, D.: Colombeau's generalized functions:
topological structures; microlocal properties.\ A simplified point of view.
Bull. Cl. Sci. Math. Nat. Sci.\ Math. \textbf{25}, 89--114 (2000)

\bibitem {Schwartz1}Schwartz, L.: Th\'{e}orie des Distributions. Hermann,
Paris (1966)

\bibitem {Travers}Travers, K.: Semilinear hyperbolic systems in one space
dimension with strongly singular initial data. Electron. J. Diff. Eqns.
\textbf{1997}(14), 1--11 (1997)
\end{thebibliography}
\end{document}